\documentclass[10pt]{amsart}

\usepackage{amsmath,amssymb,amsthm,a4wide}
\usepackage[foot]{amsaddr}
\usepackage{latexsym}
\usepackage{indentfirst}
\usepackage{graphicx}
\usepackage{placeins}
\usepackage{booktabs}
\usepackage{algorithm}
\usepackage{algpseudocode}
\usepackage{epstopdf}
\algrenewcommand\algorithmicprocedure{\textbf{function}}
\usepackage{multirow}
\usepackage{bbm}
\usepackage{bm}
\usepackage{graphicx,color}
\usepackage{diagbox}
\usepackage{hyperref}
\usepackage{braket}
\usepackage{subfigure}

\usepackage{epsfig}
\usepackage{multirow}
\usepackage{booktabs}
\usepackage[]{caption}

\usepackage{tikz}

\usepackage{enumitem}

\theoremstyle{plain}
\newtheorem{theorem}{Theorem}
\newtheorem{lemma}[theorem]{Lemma}

\theoremstyle{definition}

\theoremstyle{remark}

\newcommand{\COMMENTED}[1]{}

\newtheoremstyle{cited}%
  {3pt}
  {3pt}
  {\itshape}
  {}
  {\bfseries}
  {.}
  {.5em}
  {\thmname{#1} \thmnumber{#2} \thmnote{\normalfont#3}}

\theoremstyle{cited}

\definecolor{eggplant}{RGB}{180,33,147}

\newcommand{\abs}[1]{\left|#1\right|}
\newcommand{\bracket}[1]{\left(#1\right)}
\newcommand{\norm}[1]{\left\|#1\right\|}

\begin{document}

\author{Ziang Yu}
\address[Ziang Yu]{Committee on Computational and Applied Mathematics, University of Chicago}

\author{Shiwei Zhang}
\address[Shiwei Zhang]{Center for Computational Quantum Physics, Flatiron Institute}

\author{Yuehaw Khoo}
\address[Yuehaw Khoo]{Department of Statistics, University of Chicago}

\title{Re-anchoring Quantum Monte Carlo with Tensor-Train Sketching}

\begin{abstract}
We propose a novel algorithm for calculating the ground-state energy of quantum many-body systems by combining auxiliary-field quantum Monte Carlo (AFQMC) with tensor-train sketching. In AFQMC, a good trial wavefunction to guide the random walk is crucial for improving the sampling efficiency and controlling the sign problem. 
Our proposed method iterates between determining a new trial wavefunction in the form of a tensor train, derived from the current walkers, and using this updated trial wavefunction to anchor the next phase of AFQMC. Numerical results demonstrate that the algorithm is highly accurate for large spin systems. 
The overlap between the estimated trial wavefunction and the ground-state wavefunction also achieves high fidelity. 
We additionally provide a convergence analysis, highlighting how an effective trial wavefunction can reduce the variance in the AFQMC energy estimation. 
From a complementary perspective, our algorithm also extends the reach of tensor-train methods for studying quantum many-body systems.  
\end{abstract}

\maketitle 

\section{Introduction}

The quantum many-body problem appears in a wide range of fields, including condensed matter physics, high energy and nuclear physics, quantum chemistry, and material science. One of the most challenging parts of this problem is that the computational cost grows exponentially with the size of the system. Quantum Monte Carlo (QMC) \cite{PhysRevD.24.2278,sugiyama1986auxiliary,PhysRevB.40.506,RevModPhys.83.349} is a class of algorithms that can efficiently deal with high-dimensional problems by reducing the computational cost to a polynomial scale with the system size. It guarantees that, in expectation, one can obtain the exact ground-state energy. However, with a finite sample size, QMC generally suffers from the sign problem \cite{PhysRevB.41.9301,PhysRevLett.94.170201}, meaning the variance grows exponentially in time. To fix this issue, the constrained-path auxiliary-field quantum Monte Carlo (cp-AFQMC) \cite{zhang1997constrained} and phaseless auxiliary-field quantum Monte Carlo (ph-AFQMC) \cite{PhysRevLett.90.136401} were developed to control the sign problem, and have been successfully applied to both lattice models \cite{PhysRevB.97.235127,PhysRevA.86.053606,PhysRevLett.119.265301,PhysRevB.99.165116,PhysRevB.102.041106,PhysRevB.102.214512,PhysRevB.103.115123,PhysRevResearch.4.013239,PhysRevLett.128.203201} and realistic materials \cite{al2006auxiliary,al2007study,purwanto2009excited,purwanto2008eliminating,shee2017chemical,motta2018communication,shee2018phaseless,shee2019singlet,lee2020performance,PhysRevLett.114.226401}. In cp-AFQMC, the wavefunction is represented formally as the sum of an ensemble of statistically independent random walkers. The sign problem is avoided by introducing a trial wavefunction to bias the random walkers such that they maintain a non-negative overlap or a constant gauge 
with the trial wavefunction, with the price of a potential systematic error. This bias depends on the quality of the trial wavefunction \cite{chang2008spatially}. Ideally, the bias will be removed if the trial wavefunction is exactly the ground-state wavefunction of the system \cite{shi2021some}. Based on this result, self-consistent approaches for successively improving the trial wavefunction have been developed \cite{PhysRevB.94.235119,zheng2017stripe,PhysRevB.99.045108,PhysRevResearch.3.013065} to reduce the systematic error. 

On the other hand, a different type of approach towards the quantum many-body problem is to directly solve for the wavefunction through approximating it by some low-complexity ansatz. One of the most popular ansatz is the matrix product state (MPS) \cite{mps-1,mps-3}, also known as the tensor train (TT) \cite{tt-decomposition}. A huge advantage of TT representation is that its storage complexity and computational complexity can be made near linear (instead of exponential) in the dimension $d$. However, for complicated systems, the approximation error can be large, which leads to inaccurate 
estimates of physical observables. 

Recently, an algorithm that directly combines AFQMC and TT representation has been proposed by one of the authors \cite{chen2023combining}. In this method, the random walkers are projected into a TT by a TT-sketching technique \cite{hur2023generative} at every step. It presents a potentially cheaper strategy to update a TT, using AFQMC to simulate a matrix-vector multiplication. Although this method avoids the expensive compression of the Hamiltonian operator, it still suffers from the limitation that sometimes one cannot approximate the ground-state wavefunction as a low-rank TT.


In this paper, we provide improvements to AFQMC and tensor methods by combining the best of both worlds. A TT representation of the wavefunction is estimated periodically using TT-sketching. Then the estimated tensor-train wavefunction is used as the trial wavefunction to ``anchor'' the next episode of AFQMC. Different from a previously proposed method that uses a fixed TT trial wavefunction \cite{wouters2014projector}, our method iteratively builds the TT trial wavefunction from the AFQMC walkers. Specifically, we alternate between estimating a new trial wavefunction, and running the next episode of AFQMC. 
This can potentially remove the 
bias in cases where the sign/phase problem is present, by successively improving the trial wavefunction. 
It can also significantly reduce the variance. Our approach in estimating the tensor-train wavefunction is efficient, as it requires only enough accuracy to effectively guide the walkers.
This is unlike other tensor-train approaches, including \cite{chen2023combining}, where it is imperative to obtain an accurate approximation to the ground state since the tensor-train is used to obtain the final energy estimate. 

The rest of this paper is organized as follows. In Section~\ref{sec:preliminary}, we introduce the procedure of cp-AFQMC for calculating the ground state of spin models, and also some preliminaries for TT and TT-sketching. Our proposed algorithm is then presented in Section~\ref{sec:proposed algorithm}. In Section~\ref{sec:analysis}, we provide a theoretical analysis for the convergence of cp-AFQMC, showing that the procedure can give an accurate energy estimate, despite the fact that the wavefunction itself can have large variance. The results of numerical experiments are shown in Section~\ref{sec:numerical results}. 

\section{Preliminary}\label{sec:preliminary}

\subsection{cp-AFQMC}\label{sec:CPMC}

In this section, we illustrate the application of cp-AFQMC to determine the ground-state energy of a spin model. The ground-state problem is stated as
\begin{equation}\label{eq:ground state problem}
    \Psi_0 = \underset{\Psi\in\mathbb{R}^{2^d}}{\arg\min}\ \langle \Psi,H\Psi\rangle, \text{ s.t. }\norm{\Psi}_2^2=1, 
\end{equation}
where $d$ is the number of spins. 
The ground-state wavefunction $\Psi_0$ can be obtained by an \emph{imaginary-time evolution}, i.e.
\begin{equation}\label{eq:ground state projection}
    \Psi_0 \propto \lim_{\tau\to\infty} e^{-\tau H}\Psi_\mathrm{tr}.
\end{equation}
where $\Psi_0$ denotes the solution to \eqref{eq:ground state problem}, and $\Psi_\mathrm{tr}$ is some trial wavefunction picked in prior satisfying $\braket{\Psi_\mathrm{tr},\Psi_0}\ne0$. Numerically, the limit in \eqref{eq:ground state projection} can be calculated iteratively by
\begin{equation*}
    \Psi^{(n+1)} = e^{-\Delta\tau H}\Psi^{(n)}.
\end{equation*}
Applying $e^{-\Delta \tau H}$ to $\Psi^{(n)}$ amounts to applying a $2^d\times 2^d$ matrix to a vector of size $2^d$, which suffers from the curse of dimensionality. To solve this issue, QMC performs $e^{-\Delta \tau H}\Psi^{(n)}$ with a Monte Carlo method. In particular, cp-AFQMC represents the wavefunction as the following:
\begin{equation}\label{eq:wavefunction representation}
    \Psi^{(n)} \approx \widehat   \Psi^{(n)} = \sum_{k=1}^N   {\Phi}_k^{(n)}, 
\end{equation}
where each $ {\Phi}_k^{(n)}$, called a random walker, is a separable function, $N$ is the number of walkers. A discrete Hubbard-Stratonovich transformation \cite{ulmke2000auxiliary} can be applied to decompose the operator $e^{-\Delta\tau H}$ as  
\begin{equation}\label{eq:operator decomposition}
e^{-\Delta\tau H} = \sum_{\vec{x}}P(\vec{x}) B(\vec{x}) =  \mathbb{E}_{\vec{x}\sim P}B(\vec{x}),
\end{equation}
where $B(\vec{x}){\Phi}_k^{(n)}$ can be computed cheaply, and it remains a meanfield or separable state.
Then, one step of imaginary-time evolution can be approximated by
\begin{equation}\label{eq:imaginary time without importance}
e^{-\Delta\tau H}\widehat \Psi^{(n)}= \sum_k   e^{-\Delta\tau H}{\Phi}_k^{(n)} = \sum_k \sum_{\vec{x}}P(\vec{x}) B(\vec{x}){\Phi}_k^{(n)} \approx \sum_k {\Phi}_k^{(n+1)}=: \widehat \Psi^{(n+1)},
\end{equation}
where 
\begin{equation*}
\Phi^{(n+1)}_{k} = B(\vec{x}^{(n)}_k) {\Phi}_k^{(n)}, \vec{x}^{(n)}_k\sim P.
\end{equation*}

As key algorithmic ingredients, cp-AFQMC further incorporates importance sampling when sampling $\vec{x}_k^{(n)}$in order to reduce the variance. 
Given a walker $\Phi_k^{(n)}$ satisfying $\braket{\Psi_\mathrm{tr}, \Phi_k^{(n)}} > 0$, we incorporate a new probability distribution function $Q_k^{(n)}$, which is defined as
\begin{equation}\label{eq:importance sampling pdf normalized}
    Q_k^{(n)}(\vec{x}) = \frac{1}{\mathcal{N}_k^{(n)}} \frac{\max\{\braket{\Psi_\mathrm{tr}, B(\vec{x})\Phi_k^{(n)}},0\}}{\braket{\Psi_\mathrm{tr}, \Phi_k^{(n)}}} P(\vec{x}), 
\end{equation}
where $\mathcal{N}_k^{(n)}$ is the normalization factor. 
This new probability distribution function favors those walkers that have a large overlap with the trial wavefunction. Moreover, the walkers are guaranteed to have a positive overlap with the trial wavefunction, since the probability $Q_k^{(n)}(\vec{x})$ is set to $0$ if some $\vec{x}$ leads to a non-positive overlap with $\Psi_\mathrm{tr}$, hence constraining the paths and causing a bias, which is also known as the sign problem. In an extreme case where all the auxiliary fields $\vec{x}$ lead to non-negative overlap $\braket{\Psi_\mathrm{tr}, B(\vec{x})\Phi_k^{(n)}} \le 0$, and thus $Q_k^{(n)}(\vec{x})=0$ for all $\vec{x}$, then this walker is eliminated from the procedure.

With the new probability distribution function \eqref{eq:importance sampling pdf normalized}, the operator decomposition \eqref{eq:operator decomposition} becomes
\begin{equation}\label{eq:operator decomposition with cp approximation}
    e^{-\Delta\tau H} = \sum_{\vec{x}}P(\vec{x})B(\vec{x}) \approx \sum_{\vec{x}:Q^{(n)}_k(\vec{x})>0} P(\vec{x})B(\vec{x}) = \sum_{\vec{x}:Q^{(n)}_k(\vec{x})>0} Q^{(n)}_k(\vec{x}) \bracket{\frac{P(\vec{x})}{Q^{(n)}_k(\vec{x})} B(\vec{x})}. 
\end{equation}
The approximate equality reflects the fact that the second sum runs over all the auxiliary fields $\vec{x}$, while the final expression restricts the summation to the subset of $\vec{x}$ satisfying $Q_k^{(n)}(\vec{x})>0$. In a special case where $Q_k^{(n)}(\vec{x})>0$ for all $\vec{x}$, the approximation in \eqref{eq:operator decomposition with cp approximation} becomes exact. With this new decomposition, the propagation \eqref{eq:imaginary time without importance} becomes
\begin{equation}\label{eq:ensemble propagation math}
    e^{-\Delta\tau H}\widehat \Psi^{(n)} \approx \sum_k \sum_{\vec{x}:Q^{(n)}_k(\vec{x})>0} Q^{(n)}_k(\vec{x}) \bracket{\frac{P(\vec{x})}{Q^{(n)}_k(\vec{x})} B(\vec{x}) \Phi_k^{(n)}} \approx \sum_k \Phi_k^{(n+1)} =: \widehat \Psi^{(n+1)}, 
\end{equation}
where 
\begin{equation}\label{eq:walker propagation math}
    \Phi_k^{(n+1)} = \frac{P(\vec{x}_k^{(n)})}{Q^{(n)}_k(\vec{x}_k^{(n)})} B(\vec{x}_k^{(n)}) \Phi_k^{(n)}, \vec{x}_k^{(n)} \sim Q_k^{(n)}. 
\end{equation}

We summarize the whole algorithm in Algorithm~\ref{alg:cp-AFQMC}. 

\begin{algorithm}[!htbp]
\caption{cp-AFQMC}\label{alg:cp-AFQMC}
\begin{algorithmic}
\Require 
a Hamiltonian $H$ and a corresponding decomposition $e^{-\Delta\tau H} = \sum_{\vec{x}}P(\vec{x}) B(\vec{x})$, initial walkers $\widehat \Psi^{(0)} = \sum_k \Phi_k^{(0)}$, trial wavefunction $\Psi_\mathrm{tr}$, step size $\Delta\tau$, number of iteration steps $M$. 
\Ensure Walkers $\widehat \Psi^{(M)} = \sum_k \Phi_k^{(M)}$. 
\For{$n=0,1,\cdots,M-1$}

For walkers satisfying $\braket{\Psi_\mathrm{tr},\Phi^{(n)}_k}>0$: 
\begin{enumerate}[label=\arabic*., ref=\arabic*, leftmargin=1.5cm]
\item Sample an $\vec{x}^{(n)}_k$ from $Q_k^{(n)}$ defined in \eqref{eq:importance sampling pdf normalized}.
\item Update the walker $\Phi_k^{(n+1)} = \frac{P(\vec{x}_k^{(n)})}{Q^{(n)}_k(\vec{x}_k^{(n)})} B(\vec{x}_k^{(n)}) \Phi_k^{(n)}$. 

\end{enumerate}

\EndFor

\end{algorithmic}
\end{algorithm}

We detail the exact implementations of Algorithm~\ref{alg:cp-AFQMC} and its application on the transverse-field Ising (TFI) model in the appendix.

\subsection{Tensor-train sketching}

In this section, we introduce TT representation and TT-sketching. 
We consider a $d$-dimensional tensor
\begin{equation*}
    u: [n_1] \times \cdots \times [n_d] \to \mathbb{R}, 
\end{equation*}
where we denote $[n]=\{1,2,\cdots,n\}$ for some positive integer $n$. A TT representation is of the form
\begin{equation*}
    u(i_1,\cdots,i_d) = \sum_{\alpha_1=1}^{r_1} \cdots\sum_{\alpha_{d-1}}^{r_{d-1}} G_1(i_1,\alpha_1) G_2(\alpha_1,i_2,\alpha_2) \cdots G_{d-1}(\alpha_{d-2},i_{d-1},\alpha_{d-1}) G_d(\alpha_{d-1},i_d), 
\end{equation*}
where we call tensor $G_j$ the $j$-th tensor core of the TT, and $(r_1,\cdots,r_{d-1})$ the TT-rank. In short, we denote it as
\begin{equation*}
    u = G_1 \circ \cdots \circ G_d. 
\end{equation*}
We use the notation $m:n$ to denote the set $\{m,m+1,\cdots,n\}$ for two positive integers $m<n$, and further denote the truncated TT as
\begin{equation*}
    G_{1:j}(i_{1:j},\alpha_j) = \sum_{\alpha_1=1}^{r_1}\cdots\sum_{\alpha_{j-1}=1}^{r_{j-1}} G_1(i_1,\alpha_1) G_2(\alpha_1,i_2,\alpha_2) \cdots G_{j}(\alpha_{j-1},i_j,\alpha_j), 
\end{equation*}
and
\begin{equation*}
    G_{j:d}(i_{j:d},\alpha_{j-1}) = \sum_{\alpha_j=1}^{r_j} \cdots \sum_{\alpha_{d-1}}^{r_{d-1}} G_{j}(\alpha_{j-1},i_j,\alpha_j) \cdots G_{d-1}(\alpha_{d-2},i_{d-1},\alpha_{d-1}) G_d(\alpha_{d-1},i_d). 
\end{equation*}
A special case is when $r_1=\cdots=r_{d-1}=1$, we can write the TT as
\begin{equation*}
    u(i_1,\cdots,i_d) = G_1(i_1)\cdots G_d(i_d),
\end{equation*}
which is a separable wavefunction. 
The advantage of TT is that it represents a $d$-dimensional tensor with a storage complexity linear in $d$. Moreover, the computational cost of the inner product of two TTs is also linear in $d$. See numerical verification in the appendix. 

A natural problem is how to fit given data into a low-rank TT. In this work, we consider data given as a sum of random walkers $\widehat{\Psi}$, as in \eqref{eq:wavefunction representation}. The wavefunction in \eqref{eq:wavefunction representation} generally requires large storage, since the number of random walkers can be very large. Therefore, we aim to compress this representation into a low-rank TT. TT-sketching provides an effective procedure for accomplishing this task. 
The idea of TT-sketching is similar to randomized SVD \cite{drineas2006fast} for matrix decomposition. More specifically, if the target tensor $\widehat{\Psi}$ has a low-rank TT structure, then the unfolding matrix $\widehat{\Psi}_j \in \mathbb{R}^{(n_1\cdots n_j)\times(n_{j+1}\cdots n_d)}$, defined by reshaping the tensor $\widehat{\Psi}$:
\begin{equation*}
    \widehat{\Psi}_j(i_{1:j};i_{j+1:d}) = \widehat{\Psi}(i_1,\cdots,i_d), 
\end{equation*}
is also low-rank \cite{tt-decomposition}. Thus, we expect that 
\begin{equation*}
    \mathrm{Range}(\widehat{\Psi}_j) = \mathrm{Range}(\widehat{\Psi}_jT_j^\mathrm{sketch})
\end{equation*}
with high probability, for some chosen sketching matrix $T_j^\mathrm{sketch} = [T_\mathrm{sketch}(i_{j+1:d};\xi_j)]_{(i_{j+1:d}; \xi_j)} \in \mathbb{R}^{(n_{j+1}\cdots n_d)\times R_j}$, where $R_j$ is the number of columns of the sketching matrix, usually selected to be larger than the target rank. Furthermore, we hope to construct tensor cores $G_1, \cdots, G_j$ satisfying 
\begin{equation*}
    \mathrm{Range}(G_{1:j}) = \mathrm{Range}(\widehat{\Psi}_j), 
\end{equation*}
where we treat $G_{1:j} = [G_{1:j}(i_{1:j}; \xi_j)]_{(i_{1:j};\xi_j)} \in \mathbb{R}^{(n_1\cdots n_j)\times R_j}$ as a matrix. 
To achieve these goals, we can simply require
\begin{equation*}
    G_{1:j} = \widehat{\Psi}_j T_j^\mathrm{sketch}. 
\end{equation*}
The tensor diagram for these equations is shown in Figure~\ref{fig:tt sketching 1}, where we take a $4$-dimensional case as an illustration. We use $T = T_1 \circ T_2 \circ T_3 \circ T_4$ in TT form as the sketching tensor. 

\begin{figure}[!htbp]
\centering
\vspace{-0.2cm}
\includegraphics[width=0.6\linewidth]{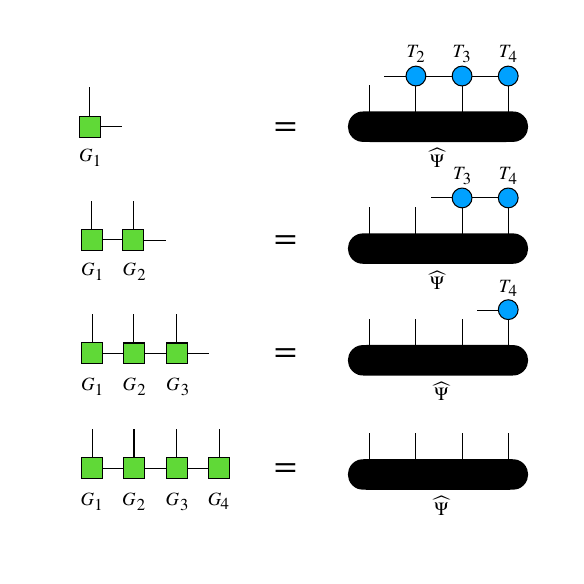}
\vspace{-0.5cm}
\caption{Tensor diagrams illustrating the goal of TT-sketching using the $4$-dimensional case as an example. }
\label{fig:tt sketching 1}
\end{figure}

The equations in Figure~\ref{fig:tt sketching 1} are over-determined and still have an exponentially large size. To further reduce computational cost, we introduce another sketching tensor $S = S_1\circ S_2\circ S_3\circ S_4$ on the left. The detailed implementation is illustrated in Figure~\ref{fig:tt sketching 2}, where each equation (except the first) can be written in the form 
\begin{equation*}
    \sum_{\xi_{j-1}}X_j(\gamma_{j-1},\xi_{j-1})G_j(\xi_{j-1},i_j,\xi_j) = Y_j(\gamma_{j-1},i_j,\xi_j). 
\end{equation*}
The tensor cores can then be obtained sequentially by solving these equations via multiplying the pseudoinverse of $X_j$ with $Y_j$. After obtaining each tensor core $G_j$, we can further trim it to the target size by inserting a projector $V_jV_j^T$, where $V_j\in\mathbb{R}^{R_j\times r_j}$ and $r_j$ is the target rank. We group $G_jV_j$ as the final tensor core, and multiply $V_j^T$ into the next tensor core.

\begin{figure}[!htbp]
\centering
\vspace{-0.2cm}
\includegraphics[width=0.6\linewidth]{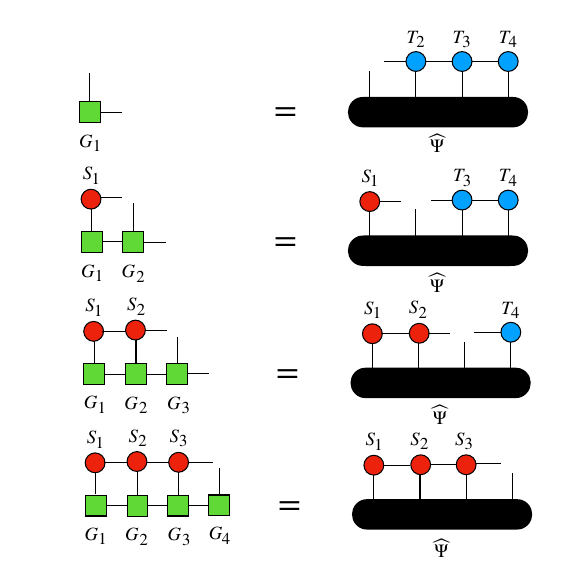}
\vspace{-0.2cm}
\caption{Tensor diagrams illustrating the detailed implementation of TT-sketching using the $4$-dimensional case as an example. }
\label{fig:tt sketching 2}
\end{figure}

For the sketching tensors $S$ and $T$, we construct their tensor cores $S_j$ and $T_j$ as
\begin{equation*}\label{eq:sketching functions core}
    S_j(\xi_{j-1},i_j,\xi_j) = \sum_l C_{S,j}(\xi_{j-1},l,\xi_j) v_l(i_j), \ T_j(\xi_{j-1},i_j,\xi_j) = \sum_l C_{T,j}(\xi_{j-1},l,\xi_j) v_l(i_j), 
\end{equation*}
where the coefficients $C_{S,j}(\xi_{j-1},l,\xi_j)$ and $C_{T,j}(\xi_{j-1},l,\xi_j)$ are Gaussian random tensors, and $\{v_l\}$ is a set of basis in the form
\begin{equation}\label{eq:cluster basis}
    v_1 = \bracket{
    1,1
    }^T, \
    v_2 = \bracket{
    \delta,-\delta
    }^T. 
\end{equation}
We aim to pick a relative small $\delta$ for the basis in \eqref{eq:cluster basis}, such that $\delta^2$ is still non-negligible but higher orders of $\delta$ can be ignored. Then 
$S$ and $T$ can be considered as $1$-cluster basis sketching functions \cite{wang2015fast,ahle2019almost,peng2023generative}. 

We remark that in this algorithm, as we sequentially compute the tensor cores, the left-hand-side tensors in Figure~\ref{fig:tt sketching 2} can also be computed sequentially, meaning that the computational costs for obtaining each tensor core are the same. Thus, the sketching algorithm has a linear time complexity in terms of the dimensionality $d$. See numerical verification in the appendix. 


\section{Proposed Algorithm}\label{sec:proposed algorithm}

In AFQMC, if 
there is a 
sign problem, the constraint $\braket{\Psi_\mathrm{tr},\Phi_k^{(n)}}>0$ may result in a systematic error, which depends on the quality of the trial function $\Psi_\mathrm{tr}$. Even in the case where there is no sign problem, a better trial wavefunction can still reduce the variance of the random walk and improve sampling efficiency.
To 
improve the trial wavefunction, 
our algorithm periodically updates the trial function with TT-sketching using the information from the walkers in cp-AFQMC. The procedure is illustrated in Figure~\ref{fig:alg illustration}. 
Formally, we summarize our proposed algorithm with details in Algorithm~\ref{alg:cp-AFQMC+TT}. 

\begin{figure}[!htbp]
\centering
\begin{tikzpicture}
    \node (A) at (0,1) {$\widehat{\Psi}^{(0)}$};
    \node (B) at (0,-1) {$\Psi_\mathrm{tr,0}$};
    \node (C) at (3.6,1) {$\widehat{\Psi}^{(M_\text{in})}$};
    \node (AC) at (1.8,1) {}; 
    \node (D) at (3.6,-1) {$\Psi_\mathrm{tr,1}$}; 
    \node (E) at (7.2,1) {$\widehat{\Psi}^{(2M_\text{in})}$};
    \node (CE) at (5.4,1) {};
    \node (F) at (7.2,-1) {$\Psi_\mathrm{tr,2}$}; 
    \node (cdot) at (8,0) {$\cdots$};

    \draw[->] (A) -- node[above] {random walk} (C);
    \draw[->] (B) -- node[right, xshift=6pt, yshift=6pt] {cp approx.} (AC);
    \draw[->] (C) -- node[left, yshift=-10pt] {sketching} (D); 
    \draw[->] (C) -- node[above] {random walk} (E);
    \draw[->] (D) -- node[right, xshift=6pt, yshift=6pt] {cp approx.} (CE);
    \draw[->] (E) -- node[left,yshift=-10pt] {sketching} (F); 
\end{tikzpicture}
\vspace{-0.3cm}
\caption{Proposed framework: cp-AFQMC-re-anchoring. }
\label{fig:alg illustration}
\end{figure}

\begin{algorithm}[!htbp]
\caption{cp-AFQMC-re-anchoring}\label{alg:cp-AFQMC+TT}
\begin{algorithmic}
\Require 
A Hamiltonian $H$ and a corresponding decomposition $e^{-\Delta\tau H} = \sum_{\vec{x}}P(\vec{x}) B(\vec{x})$, initial trial wavefunction $\Psi_{\mathrm{tr},0}$, step size $\Delta\tau$, number of iterations $M_\text{out}$ and $M_\text{in}$. 
\Ensure 
An ensemble of walkers $\widehat \Psi^{(M_\text{in} M_\text{out})} = \sum_k \Phi^{(M_\text{in} M_\text{out})}_{k}$, a wavefunction $\Psi_{\mathrm{tr},M_\text{out}}$.

\State Let $\widehat \Psi^{(0)} = \sum_k \Psi_\mathrm{tr,0}$. 
\For{$n_\text{out}=0,1,\cdots,M_\text{out}-1$}
\begin{itemize}
\item $ \widehat \Psi^{(M_\text{in}n_\text{out}+M_\text{in})}$ = cp-AFQMC($H,\widehat\Psi^{(M_\text{in}n_\text{out})}, \Psi_{\mathrm{tr},n_\text{out}},\Delta \tau, M_\text{in}$).
\item 
Apply TT-sketching on $\widehat{\Psi}^{(M_\text{in}n_\text{out}+M_\text{in})}$ and get a new trial wavefunction $\Psi_{\mathrm{tr},n_\text{out}+1}$.
\end{itemize}
\EndFor

\end{algorithmic}
\end{algorithm}

\section{Convergence Analysis of cp-AFQMC}\label{sec:analysis}

In this section, we provide some convergence analysis for cp-AFQMC. In Section~\ref{section:variance control}, we first understand the effect of importance sampling in controlling the variance of the overlap of the walkers with the trial wavefunction in one step of cp-AFQMC. In Section~\ref{section:convergence of cp-afqmc}, we then use this to analyze the convergence of cp-AFQMC in terms of both the energy and the wavefunction itself. In Section~\ref{sec:convergence energy}, we prove that a better trial wavefunction can lead to a smaller variance in ground-state energy. This fact motivates our algorithm to obtain a high-quality trial wavefunction, so that we effectively reduce the variance in ground-state energy estimation within the cp-AFQMC procedure. 
In Section~\ref{sec:convergence walker}, we demonstrate that the random walkers 
individually exhibit a large variance, which cannot be eliminated by improving the trial wavefunction. Consequently, the
finite ensemble of walkers \eqref{eq:wavefunction representation} does not reliably converge to the ground-state wavefunction in practice. This observation further motivates our proposed algorithm: we extract an explicit wavefunction in TT form, approximating the ground-state wavefunction, from the noisy data produced by the cp-AFQMC procedure.

\subsection{Variance and bias per step of cp-AFQMC}\label{section:variance control}
In this section, we provide an explanation on how importance sampling controls the variance of the overlaps in one step of cp-AFQMC, and discuss the systematic error caused by constrained-path approximation. For convenience, we rewrite the walker propagation \eqref{eq:walker propagation math} as
\begin{equation*}
    \Phi_k^{(n+1)} := \widetilde{B}_k^{(n)}(\vec{x}_k^{(n)}) \Phi_k^{(n)}, \vec{x}_k^{(n)} \sim Q_k^{(n)}, 
\end{equation*}
where we denote the operator
\begin{equation}\label{eq:tilde B}
    \widetilde{B}_k^{(n)}(\vec{x}) = \frac{P(\vec{x})}{Q_k^{(n)}(\vec{x})} B(\vec{x}), 
\end{equation}
which is only well-defined when $Q_k^{(n)}(\vec{x})>0$. Based on \eqref{eq:ensemble propagation math}, we hope 
\begin{equation*}
    \widetilde{B}_k^{(n)}(\vec{x}) \Phi_k^{(n)} \approx e^{-\Delta\tau H}\Phi_k^{(n)}
\end{equation*}
in a certain sense. We firstly state an observation regarding the update $\widetilde{B}_k^{(n)}(\vec{x}) \Phi_k^{(n)}$. 

\begin{lemma}\label{lemma:zero-variance}
$\langle \Psi_\mathrm{tr}, \widetilde{B}_k^{(n)}(\vec{x}) \Phi_k^{(n)}\rangle $ is a constant in $\vec{x}$. 
\end{lemma}

\begin{proof}
Suppose we have a current walker $\Phi_k^{(n)}$, the overlap of the walker at the next step is then
    \begin{equation*}
         \braket{\Psi_\mathrm{tr},\widetilde{B}_k^{(n)}(\vec{x})\Phi_k^{(n)}} = \left\langle\Psi_\mathrm{tr},\frac{P(\vec{x})}{Q_k^{(n)}(\vec{x})} B(\vec{x})\Phi_k^{(n)}\right\rangle = \braket{\Psi_\mathrm{tr},\Phi_k^{(n)}}\mathcal{N}_k^{(n)}.
    \end{equation*}
The last equality is due to  \eqref{eq:importance sampling pdf normalized}, which shows that $\braket{\Psi_\mathrm{tr},\widetilde{B}_k^{(n)}(\vec{x})\Phi_k^{(n)}}$ has no $\vec{x}$ dependency since $\mathcal{N}_k^{(n)}$ has no $\vec{x}$ dependency.
\end{proof}
Lemma~\ref{lemma:zero-variance} tells us that the variance of $ \braket{\Psi_\mathrm{tr},\widetilde{B}_k^{(n)}(\vec{x})\Phi_k^{(n)}}$  is reduced by importance sampling, in fact, to $0$. We then study the value of $\braket{\Psi_\mathrm{tr},\widetilde{B}_k^{(n)}(\vec{x})\Phi_k^{(n)}}$, as a consequence of Lemma~\ref{lemma:zero-variance},  in the following two lemmas.
\begin{lemma}\label{lemma: no sign problem zero variance}
    When there is no sign problem, i.e. $ \braket{\Psi_\mathrm{tr},B(\vec{x})\Phi_k^{(n)}}>0$ for all $\vec{x}$, then
    \begin{equation}\label{eq:lemma: no sign problem zero variance}
         \braket{\Psi_\mathrm{tr},\widetilde{B}_k^{(n)}(\vec{x})\Phi_k^{(n)}} = 
         \braket{\Psi_\mathrm{tr}, e^{-\Delta\tau H}\Phi_k^{(n)} }. 
    \end{equation}
\end{lemma}
\begin{proof}
When $\braket{\Psi_\mathrm{tr},\widetilde{B}_k^{(n)}(\vec{x})\Phi_k^{(n)}}>0$, then $Q_k^{(n)}(\vec{x})>0$ for all $\vec{x}$ (see \eqref{eq:importance sampling pdf normalized}).  In this case, $\widetilde{B}_k^{(n)}(\vec{x})$ is well-defined for every $\vec{x}$. Then
\begin{equation*}
 \braket{\Psi_\mathrm{tr},\widetilde{B}_k^{(n)}(\vec{x})\Phi_k^{(n)}} =  \braket{\Psi_\mathrm{tr},\sum_{\vec{x}} Q_k^{(n)}(\vec{x})\widetilde{B}_k^{(n)}(\vec{x})\Phi_k^{(n)}} = \braket{\Psi_\mathrm{tr},\sum_{\vec{x}} P(\vec{x})B(\vec{x})\Phi_k^{(n)}} =  \braket{\Psi_\mathrm{tr}, e^{-\Delta\tau H}\Phi_k^{(n)} },
\end{equation*}
where the first equality is due to Lemma~\ref{lemma:zero-variance} and the fact that $\widetilde{B}_k^{(n)}(\vec{x})$ can be defined for every $\vec{x}$, and the second and third equality follow from  \eqref{eq:tilde B} and \eqref{eq:operator decomposition}, respectively.
\end{proof}
It might be the case where $Q_k^{(n)}(\vec{x})=0$ for some auxiliary fields $\vec{x}$. In this situation, some of the $\vec{x}$ values cannot be reached. Hence, a systematic error on $\braket{\Psi_\mathrm{tr},\widetilde{B}_k^{(n)}(\vec{x})\Phi_k^{(n)}}$ is introduced, as explained in the following.

\begin{lemma}
      When there is a sign problem, i.e. $ \braket{\Psi_\mathrm{tr},B(\vec{x})\Phi_k^{(n)}} \le 0$ for some $\vec{x}$, then 
    \begin{equation*}
         \braket{\Psi_\mathrm{tr},\widetilde{B}_k^{(n)}(\vec{x})\Phi_k^{(n)}} -
         \braket{\Psi_\mathrm{tr}, e^{-\Delta\tau H}\Phi_k^{(n)} } = -\sum_{\vec{x}:Q_k^{(n)}(\vec{x})=0} P(\vec{x}) \braket{\Psi_\mathrm{tr},B(\vec{x})\Phi_k^{(n)}} \ge 0. 
    \end{equation*}
\end{lemma}

\begin{proof}
    When $ \braket{\Psi_\mathrm{tr},B(\vec{x})\Phi_k^{(n)}}\leq 0$ for some $\vec{x}$, there are values of $\vec{x}$ such that $Q_k^{(n)}(\vec{x})=0$. In this case, \eqref{eq:operator decomposition with cp approximation} can be expressed more explicitly as
    \begin{equation*}
    \begin{split}
        e^{-\Delta\tau H} &= \sum_{\vec{x}:Q_k^{(n)}(\vec{x})=0} P(\vec{x})B(\vec{x}) + \sum_{\vec{x}:Q_k^{(n)}(\vec{x})>0} P(\vec{x}) B(\vec{x}) \\
        &= \sum_{\vec{x}:Q_k^{(n)}(\vec{x})=0} P(\vec{x})B(\vec{x}) + \sum_{\vec{x}:Q_k^{(n)}(\vec{x})>0} Q_k^{(n)}(\vec{x}) \widetilde{B}_k^{(n)}(\vec{x}). 
    \end{split}
    \end{equation*}
Then
\begin{equation*}
\begin{split}
     &\quad  \braket{\Psi_\mathrm{tr}, e^{-\Delta\tau H}\Phi_k^{(n)}} \\
    &= \braket{\Psi_\mathrm{tr}, \sum_{\vec{x}:Q_k^{(n)}(\vec{x})=0} P(\vec{x})B(\vec{x}) \Phi_k^{(n)} + \sum_{\vec{x}:Q_k^{(n)}(x)>0} Q_k^{(n)}(\vec{x})\widetilde{B}_k^{(n)}(\vec{x})\Phi_k^{(n)}} \\
    &= \sum_{\vec{x}:Q_k^{(n)}(\vec{x})=0} P(\vec{x}) \braket{\Psi_\mathrm{tr},B(\vec{x}) \Phi_k^{(n)}} + \sum_{\vec{x}:Q_k^{(n)}(\vec{x})>0} Q_k^{(n)} (\vec{x})\braket{\Psi_\mathrm{tr},\widetilde{B}_k^{(n)}(\vec{x})\Phi_k^{(n)}} \\
    &= \sum_{\vec{x}:Q_k^{(n)}(\vec{x})=0} P(\vec{x}) \braket{\Psi_\mathrm{tr},B(\vec{x}) \Phi_k^{(n)}} +   \braket{\Psi_\mathrm{tr},\widetilde{B}_k^{(n)}(\vec{x})\Phi_k^{(n)}}
\end{split}
\end{equation*}
where the last equality follows from Lemma~\ref{lemma:zero-variance}. 
\end{proof}

In this case, the actual inner product $\braket{\Psi_\mathrm{tr},\widetilde{B}_k^{(n)}(\vec{x})\Phi_k^{(n)}}$ is larger than the expected value $\braket{\Psi_\mathrm{tr}, e^{-\Delta\tau H}\Phi_k^{(n)} }$, thus a bias towards the trial function $\Psi_\mathrm{tr}$ is introduced, which causes the sign problem.

\subsection{Convergence analysis}\label{section:convergence of cp-afqmc}
In this section, we present how importance sampling with a trial wavefunction facilitates the convergence of energy. We focus on the case where there is no sign problem as in Lemma~\ref{lemma: no sign problem zero variance}. We want to show that in this, having a trial wavefunction that is closer to the ground-state wavefunction gives a smaller variance on the energy estimator (Section~\ref{sec:convergence energy}). We further show in Section~\ref{sec:convergence walker} that unlike a typical power method, we cannot expect that the walkers converge to the ground-state wavefunction as the variance of the
individual walkers, 
controlled primarily by the chosen form of $B(\vec x)$, is too large. 

The results give us a better understanding of cp-AFQMC. Although it originates from the imaginary-time evolution \eqref{eq:ground state problem}, the wavefunction itself, represented as \eqref{eq:wavefunction representation}, remains a \emph{statistical} sampling 
of
the ground-state wavefunction. 
In general, it does not provide a direct wavefunction of well-controlled variance \cite{purwanto2004-brute-force-estimator}. 
However, we can still extract useful information out of it, for example, the ground-state energy. 

Suppose the Hamiltonian $H$ has eigenvalues $E_0 < E_1 \le E_2 \le \cdots \le E_{2^d-1}$, where we have assumed that the eigenspace of the smallest eigenvalue $E_0$ is non-degenerate. This guarantees that there is a unique lowest energy eigenvector. In this cae, the eigen-decomposition of $H$ is denoted as
\begin{equation*}
    H = E_0 \Psi_0 \Psi_0^T + E_1 \Psi_1 \Psi_1^T + \cdots + E_{2^d-1} \Psi_{2^d-1} \Psi_{2^d-1}^T, 
\end{equation*}
Here $\Psi_0$ is the ground state, and $\Psi_i$'s are excited states. The trial function can be decomposed using the eigenstates as 
\begin{equation}\label{eq:trial decompose}
    \Psi_\mathrm{tr} = c_0 \Psi_0 + c_\perp \Psi_\perp, 
\end{equation}
of which the second term $c_\perp \Psi_\perp$ represents the component of excited states, with $\Psi_\perp$ being a unit vector orthogonal to $\Psi_0$. The quantity $\abs{c_\perp / c_0}$ then characterizes how close $\Psi_\mathrm{tr}$ is to the ground state $\Psi_0$. In what follows, we study how the convergence of the energy and the walkers depends on $\abs{c_\perp / c_0}$. 

\subsubsection{Convergence analysis of the energy}\label{sec:convergence energy}

In this section, we analyze how the mixed energy estimator converges to the ground-state energy in one step of cp-AFQMC. We want to show that the variance of cp-AFQMC is controlled by $\abs{c_\perp/c_0}$. When $\abs{c_\perp/c_0}$ gets smaller, the variance of the energy gets smaller, i.e. the convergence of cp-AFQMC behaves like a vanilla power iteration. In an ideal case where $\abs{c_\perp/c_0}=0$, i.e., the trial function is exactly the ground-state wavefunction $\Psi_\mathrm{tr}=\Psi_0$, there is a zero-variance principle. Thus, we are motivated to improve the trial wavefunction in Algorithm~\ref{alg:cp-AFQMC+TT} to reduce the variance in ground-state energy estimation. 


We first state a lemma that characterizes the propagation of a walker when applying the exact operator $e^{-\Delta\tau H}$ instead of $\widetilde{B}_k^{(n)}(\vec{x})$.

\begin{lemma}\label{lemma:energy error decreasing new}
We further decompose the trial wavefunction as 
\begin{equation*}
    \Psi_\mathrm{tr} = c_0 \Psi_0 + c_1 \Psi_1 + \cdots + c_{2^d-1} \Psi_{2^d-1}, 
\end{equation*}
and decompose the walker as
\begin{equation}\label{eq:phi_k^n expansion}
    \Phi_k^{(n)} = c_0^\prime \Psi_0 + c_1^\prime \Psi_1 + \cdots + c_{2^d-1}^\prime \Psi_{2^d-1}. 
\end{equation}
Assume $c_0c_0^\prime \ne 0$. In addition, we assume that the trial wavefunction satisfies:
\begin{equation}\label{eq:assumption c_0 dominates}
    a_0 := \abs{\frac{c_\perp}{c_0}} = \frac{\sqrt{\sum_{j=1}^{2^d-1} c_j^2}}{\abs{c_0}} < 1,
\end{equation}
and
\begin{equation}\label{eq:assmption c_1 dominates}
    a_1 := \frac{\sqrt{\sum_{j=2}^{2^d-1} c_j^2(E_j-E_0)^2}}{c_1(E_1-E_0)} < 1. 
\end{equation}
Furthermore, for the walker $\Phi_k^{(n)}$, we denote
\begin{equation}\label{eq:walker angle}
\tan \theta_k^{(n)} :=\frac{\sqrt{\sum_{j=1}^{2^d-1}{c_j^\prime}^2}}{\vert c'_0\vert},\quad \tan \beta_k^{(n)} :=\frac{\sqrt{\sum_{j=2}^{2^d-1}{c_j^\prime}^2}}{\vert c'_1\vert}.
\end{equation}
Then, 
\begin{equation}\label{eq:lemma energy error decrease in exact power iteration}
    \abs{  \frac{\braket{\Psi_\mathrm{tr},He^{-\Delta\tau H}\Phi_k^{(n)}}}{\braket{\Psi_\mathrm{tr},e^{-\Delta\tau H}\Phi_k^{(n)}}}  - E_0} \le e^{-\Delta\tau(E_1-E_0)} \frac{1+a_0\tan\theta_k^{(n)}}{1-a_0\tan\theta_k^{(n)}}\frac{1+a_1\tan\beta_k^{(n)}}{1-a_1\tan\beta_k^{(n)}}\abs{  \frac{\braket{\Psi_\mathrm{tr},H\Phi_k^{(n)}}}{\braket{\Psi_\mathrm{tr},\Phi_k^{(n)}}}  - E_0}. 
\end{equation}
\end{lemma}

With this result, we can then analyze how the energy changes when we apply the actual propagator $\widetilde{B}_k^{(n)}(\vec{x})$ to a walker $\Phi_k^{(n)}$. 

\begin{theorem}\label{thm:energy convergence}
    Suppose there is no sign problem, as defined in Lemma~\ref{lemma: no sign problem zero variance}. With the same assumptions as Lemma~\ref{lemma:energy error decreasing new},
    \begin{multline}\label{eq:theorem energy convergence}
      \abs{\frac{\braket{\Psi_\mathrm{tr},H\widetilde{B}_k^{(n)}(\vec{x})\Phi_k^{(n)}}}{\braket{\Psi_\mathrm{tr},\widetilde{B}_k^{(n)}(\vec{x})\Phi_k^{(n)}}} -E_0} =e^{-\Delta\tau(E_1-E_0)} \frac{1+a_0\tan\theta_k^{(n)}}{1-a_0\tan\theta_k^{(n)}}\frac{1+a_1\tan\beta_k^{(n)}}{1-a_1\tan\beta_k^{(n)}}\abs{  \frac{\braket{\Psi_\mathrm{tr},H\Phi_k^{(n)}}}{\braket{\Psi_\mathrm{tr},\Phi_k^{(n)}}}  - E_0}  \\
       + \|H-E_0 I \|_2\frac{\|\widetilde{B}_k^{(n)}(\vec{x})-e^{-\Delta\tau H} \|_2}{e^{-\Delta\tau E_0}}  \frac{a_0\tan\theta_k^{(n)}}{1+a_0 e^{-\Delta(E_1-E_0)}\tan\theta_k^{(n)}}.
    \end{multline}
\end{theorem}
Proof of both Lemma~\ref{lemma:energy error decreasing new} and Theorem~\ref{thm:energy convergence} can be found in the appendix. 
We notice that the energy convergence consists of a deterministic power iteration part and a fluctuating part that is roughly $a_0\|H-E_0 I \|_2 \tan\theta_k^{(n)} = \abs{c_\perp/c_0}\|H-E_0 I \|_2 \tan\theta_k^{(n)}$ (when $a_0$ is small). Here $\|H-E_0 I \|_2$ resembles the energy scale, and $\tan\theta_k^{(n)}$ is a random variable that captures the deviation of a random walker from the ground state. While currently, we do not have a characterization of the variance of $\tan\theta_k^{(n)}$, if we assume that it has a bounded variance of $O(1)$, when $a_0$ gets small, i.e. the trial function gets closer to the ground-state wavefunction, the fluctuation term can be made small. In an ideal case where $a_0=0$ so that the trial function is exactly the ground-state wavefunction $\Psi_\mathrm{tr}=\Psi_0$, there is no fluctuation in the energy estimation. As a matter of fact, we have
\begin{equation*}
    \frac{\braket{\Psi_0,H\Psi}}{\braket{\Psi_0,\Psi}} = E_0, 
\end{equation*}
for any wavefunction $\Psi$.

\subsubsection{Convergence analysis of the walkers}\label{sec:convergence walker}

Now, we analyze the convergence of a particular walker $\Phi_k^{(n)}$ to the ground state, by keeping track of the angle between $\Phi_k^{(n)}$ and the ground state $\Psi_0$. As shown later in Theorem~\ref{thm: walker convergence}, unlike the results in energy convergence, even if the trial function is exactly the ground-state wavefunction, there is still an error term that does not vanish. 
In other words, the wavefunction represented by the ensemble of the walkers \eqref{eq:wavefunction representation} has a large variance, and does not converge to
the ground state reliably, regardless of the 
trial wavefunction. 
The convergence to the 
ground state is only in the sense of Monte Carlo sampling, as a statistical ensemble. Our algorithm addresses this limitation by reconstructing a wavefunction in TT form using information extracted from the random walkers. 

We denote the tangent of the angle between $\Phi_k^{(n)}$ and $\Psi_0$ as $\tan\theta_k^{(n)}$ as in \eqref{eq:walker angle}. Similarly, the tangent of the angle between $\widetilde{B}_k^{(n)}(\vec{x})\Phi_k^{(n)}$ ($\widetilde{B}_k^{(n)}(\vec{x})$ is defined in \eqref{eq:tilde B}) and $\Psi_0$ is denoted as $\tan \theta_k^{(n+1)}(\vec{x})$. Our goal is to show that there is \emph{no contraction} in terms of this angle in 1-step of cp-AFQMC. To make the notations more concise, we rewrite the angle defined in \eqref{eq:walker angle} as 
\begin{equation*}
\tan \theta_k^{(n+1)}(\vec{x}) =  \frac{\|\mathcal{P}_\perp \widetilde{B}_k^{(n)}\Phi_k^{(n)} \|}{\vert\braket{\Psi_0,\widetilde{B}_k^{(n)}(\vec{x})\Phi_k^{(n)}}\vert},\quad \tan \theta_k^{(n)} =  \frac{\|\mathcal{P}_\perp \Phi_k^{(n)} \|}{\vert\braket{\Psi_0,\Phi_k^{(n)}}\vert}, 
\end{equation*}
where $\mathcal{P}_\perp$ is the orthogonal projector onto the space spanned by $\{\Psi_1,\ldots \Psi_{2^d-1}\}$. We want to show in the following that $\tan \theta_k^{(n+1)}(\vec{x})<\tan \theta_k^{(n)}$ up to some noise terms. However, the noise terms will not vanish even for very good trial wavefunction. 

\begin{theorem}\label{thm: walker convergence}
    Suppose there is no sign problem, as defined in Lemma~\ref{lemma: no sign problem zero variance}. 
With the same assumptions as Lemma~\ref{lemma:energy error decreasing new}, we further assume that
\begin{equation*}
    a_0 < e^{-\Delta\tau E_0} \frac{\vert{\braket{\Psi_0,\Phi_k^{(n)}}}\vert}{\vert{\braket{\Psi_\perp,(\widetilde{B}_k^{(n)}(\vec{x}) - e^{-\Delta\tau H}) \Phi_k^{(n)}}\vert}}, 
\end{equation*}
Then, 
\begin{multline*}
\tan \theta_k^{(n+1)}(\vec{x}) \leq  e^{-\Delta \tau (E_1-E_0)}\tan \theta_k^{(n)} + e^{\Delta \tau E_0}\left(a_0 \tan \theta_k^{(n)}  +1\right)O\left( \frac{ \|\widetilde{B}_k^{(n)}(\vec{x}) - e^{-\Delta\tau H}\|}{\vert\braket{\Psi_0,\Phi_k^{(n)}}\vert}  \right)\\
    \quad + a_0 e^{2\Delta \tau E_0}O\left( \frac{ \|\widetilde{B}_k^{(n)}(\vec{x}) - e^{-\Delta\tau H}\|}{\vert\braket{\Psi_0,\Phi_k^{(n)}}\vert}  \right)^2. 
\end{multline*}

\end{theorem}

Proof of Theorem~\ref{thm: walker convergence} can be found in the appendix. 
As we can see, even when $a_0=0$, i.e. the trial wavefunction is exactly the ground-state wavefunction, there is still an error term $O\left(\frac{ \|\widetilde{B}_k^{(n)}(\vec{x}) - e^{-\Delta\tau H}\|}{\vert\braket{\Psi_0,\Phi_k^{(n)}}\vert}\right)$. Thus, the individual walker is not guaranteed to get closer to the ground state. 
In fact, in general, one needs an exponential number of walkers to get $\frac{ \|\widetilde{B}_k^{(n)}(\vec{x}) - e^{-\Delta\tau H}\|}{\vert\braket{\Psi_0,\Phi_k^{(n)}}\vert}<1$. This is due to the fact that, generally, one needs to sample an exponential number of $\widetilde{B}_k^{(n)}(\vec{x})$ to approximate $\exp(-\Delta \tau H)$. 
In other words, when sampling a multi-dimensional function, we can evaluate various properties with a polynomial cost, but to map out the function to pre-specified details by relying on the sample positions can require exponential cost.

\section{Numerical Results}\label{sec:numerical results}

In this section, we apply our proposed algorithm on transverse-field Ising (TFI) systems \cite{fisher1998distributions,chakrabarti2008quantum}. The Hamiltonian of the TFI model is given by
\begin{equation}\label{eq:TFI hamiltonian}
    H = -g\sum_{i=1}^d \sigma_i^x - \sum_{\langle i,j \rangle} \sigma_i^z\sigma_j^z := H_1 + H_2, 
\end{equation}
where $g$ is the strength of the magnetic field, $d$ is the total number of spins, $\langle i,j \rangle$ represents pairs of nearest neighbors, $\sigma^x,\sigma^z\in\mathbb{R}^{2\times2}$ are Pauli matrices \cite{gull1993imaginary}, and we denote
\begin{equation*}
    \sigma_i^{x,z} = I_2 \otimes \cdots \otimes \underbrace{\sigma^{x,z}}_{i\text{-th term}} \otimes \cdots \otimes I_2 \in \mathbb{R}^{2^d\times 2^d}. 
\end{equation*}
For more details about TFI system, we refer the reader to the appendix. 
For the TFI system, 
the cp-AFQMC can be made 
free of sign problem; a detailed analysis is provided in the supplementary appendix. 
Without a sign problem, the constrained-path approximation will not incur 
any systematic error, 
and the trial wavefunction $\Psi_\mathrm{tr}$ purely introduces importance sampling. Depending on the quality of $\Psi_\mathrm{tr}$, the efficiency as well as other controllable biases in the computation are affected, including those from 
fixed time step $\Delta\tau$ and 
finite population size $N$. See further discussions in the appendix. 
These provide an excellent test ground to illustrate and study 
our algorithm.
As we show below, the algorithm leads to 
significant improvement across the board 
by allowing better importance sampling. 
Furthermore, by combining the features of AFQMC 
and the density matrix renormalization group (DMRG) \cite{PhysRevLett.69.2863,fishman2022itensor}, the algorithm shows advantages and
extends the reach beyond  each of its 
components separately.


\subsection{cp-AFQMC-re-anchoring for transverse-field Ising model}\label{sec:cp-afqmc on tfi}

In this section, we mainly compare our proposed algorithm, cp-AFQMC-re-anchoring, with vanilla cp-AFQMC for both 1D and 2D TFI systems, and illustrate that cp-AFQMC-re-anchoring significantly improves the 
accuracy and efficiency.

In all of our experiments, the step size of the imaginary-time evolution is set to $\Delta\tau=0.01$. Unless otherwise noted, we have verified that the remaining Trotter error is smaller than our statistical error bar or the signal for comparison.
We set the initial trial wavefunction to be the fully disordered state
\begin{equation}\label{eq:initial trial}
    \Psi_\mathrm{tr,0} = \frac{1}{\sqrt{2^d}} \cdot
    \bracket{
    1,1
    }^T \otimes \cdots \otimes 
    \bracket{
    1,1
    }^T, 
\end{equation}
which can also be interpreted as the eigenvector corresponding to the smallest eigenvalue of $H_1$ (The first term of the Hamiltonian in \eqref{eq:TFI hamiltonian}). 
We implement TT-sketching to update the trial wavefunction every $50$ steps. At the $2000$-th step, we stop updating the trial wavefunction and keep using the latest one. We set the number of walkers to scale linearly with the number of spins. For systems with $16,32,64$ and $96$ spins, the numbers of walkers are set to $2000,4000,8000$ and $12000$, respectively. 

In the TT-sketching procedure, we take $\delta=0.1$ in \eqref{eq:cluster basis} for the sketching functions. For systems with $16$ or $32$ spins, the rank of the sketching functions is set to $60$. For systems with $64$ or $96$ spins, we set the rank to be $150$. The target rank of the sketched TT is set to $r=4$ in all experiments.





The 1D transverse-field Ising model with periodic boundary undergoes a quantum phase transition \cite{sachdev1999quantum} at $g=1$ (see \eqref{eq:TFI hamiltonian}), which is demonstrated in Figure~\ref{fig:phase transition} using the 1D system with $16$ spins as an example. To be more specific, We run our algorithm on the 1D system with $16$ spins with different transverse fields, and measure the ground-state energy $E_0(g)$. We also use Hellmann-Feynman theorem \cite{politzer2018hellmann} to calculate the energy of the transverse field by
\begin{equation*}
    \frac{\braket{\Psi_0(g),(-\sum_{i=1}^d\sigma_i^x)\Psi_0(g)}}{\braket{\Psi_0(g),\Psi_0(g)}} = \frac{\mathrm{d}}{\mathrm{d}g} E_0(g), 
\end{equation*}
where $\Psi_0(g)$ represents the ground-state wavefunction of the TFI system with the transverse-field strength being $g$. 
Numerically, this derivative is approximated by finite difference: 
\begin{equation*}
    \frac{\mathrm{d}}{\mathrm{d}g} E_0(g) \approx \frac{E_0(g+\Delta g) - E_0(g)}{\Delta g}, 
\end{equation*}
where we pick $\Delta g=0.01$ in practice. 

\begin{figure}[!htbp]
\centering
\subfigure[Ground-state energy]{
\begin{minipage}{0.49\textwidth}
    \centering
    \includegraphics[height=4cm, width=6cm]{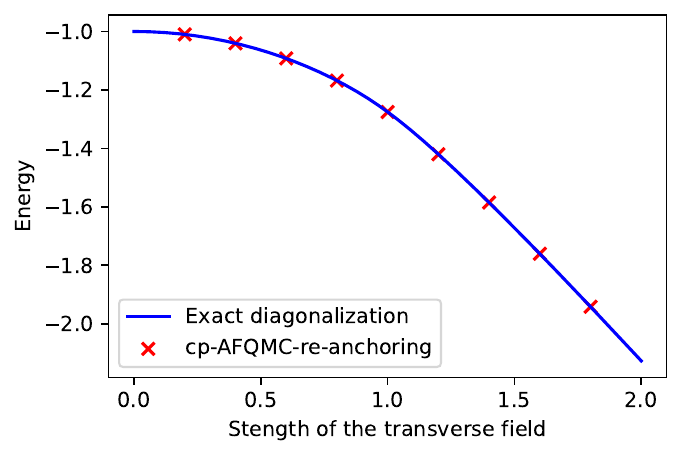}
    \vspace{0.2cm}
\end{minipage}
}
\hspace{-1.5cm}
\subfigure[Transverse-field energy]{
\begin{minipage}{0.49\textwidth}
    \centering
    \includegraphics[height=4cm, width=6cm]{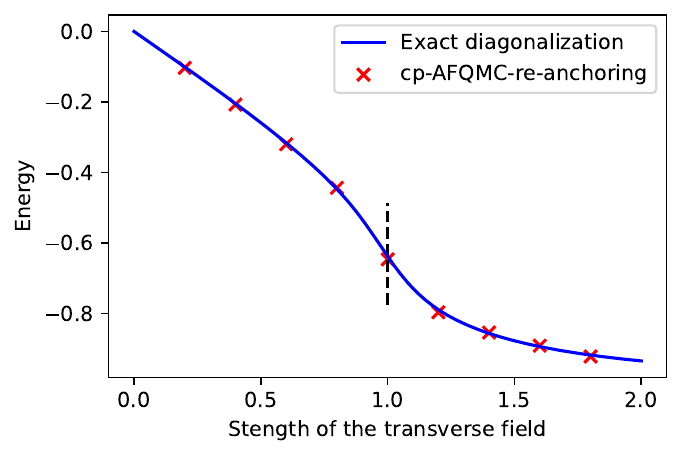}
    \vspace{0.2cm}
\end{minipage}
}
\vspace{-0.25cm}
\caption{Results of cp-AFQMC-re-anchoring for $16$ spins in 1D with different transverse-field strength $g$. (a): ground-state energy; (b): transverse-field energy, where the critical point $g=1$ is illustrated with the black dashed line. }
\label{fig:phase transition}
\end{figure}

Since the quantum phase transition at the critical point $g=1$ makes it the most difficult case to deal with, we also compare cp-AFQMC-re-anchoring and the vanilla cp-AFQMC on larger systems, including 1D system with $d=32,64$ and $96$ spins, at $g=1$. The ground-state energy of 1D periodic system can be calculated analytically \cite{pfeuty1970one}, and the ground-state wavefunction can be well approximated by DMRG. 

In addition, we also test a quasi-1D system, where the spins are arranged on a $4\times16$ lattice. We incorporate periodic boundary on the $4$-site axis, and open boundary on the $16$-site axis. Thus, it appears like a cylinder. The transverse field is also set to $g=1$. Despite the fact that the spins are arranged on a 2D lattice, the ground-state energy and ground-state wavefunction can still be well approximated by DMRG, due to the specific type of boundary conditions. Thus, we still use DMRG result as a reference. 

As illustrated in Figure~\ref{fig:energy}, for each system, the energy obtained by vanilla cp-AFQMC is of the order $O(10^{-3})$, while cp-AFQMC-re-anchoring can achieve the accuracy of order $O(10^{-5})$ or even smaller. The relative error of the energy is shown in Table~\ref{tab:error}. In addition, the overlap between our sketched trial wavefunction and the ground-state wavefunction is shown in Figure~\ref{fig:overlap}. For 1D systems, the sketched trial wavefunction in the form of a rank-$4$ TT can approximate the ground-state wavefunction quite well. For systems with $32$, $64$ and $96$ spins, the overlap can achieve around $0.9$, $0.8$, and $0.7$, respectively. For the $4\times 16$ cylinder system, the overlap converges to around $0.88$. 
As the quality of the trial wavefunction is improved, all sources of error in AFQMC will be reduced, resulting in smaller statistical errors in the estimators and reduced bias (population control and Trotter). This is clearly seen in the results in Table~\ref{tab:error}.

\begin{figure}[!htbp]
\centering
\subfigure[]{
\begin{minipage}{0.24\textwidth}
    \centering
    \includegraphics[width=\linewidth]{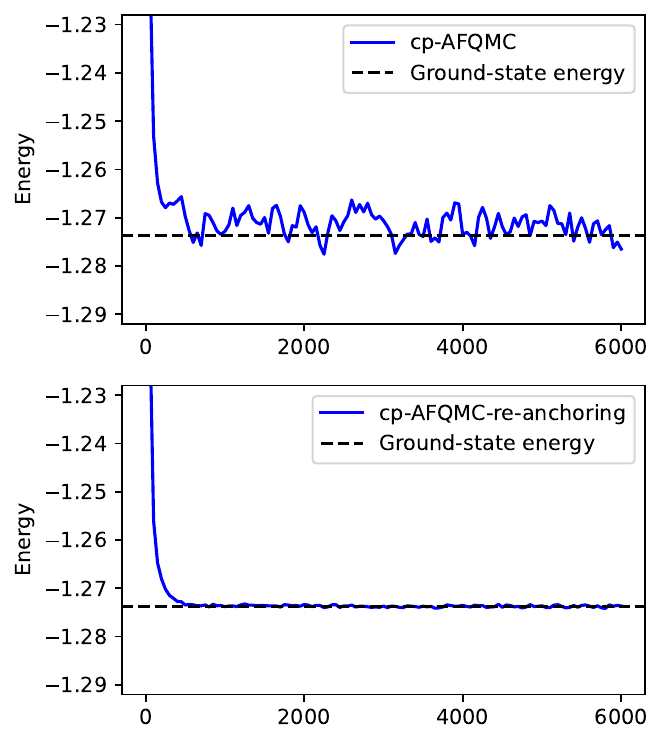}
    \vspace{-0.4cm}
\end{minipage}
}
\hspace{-0.5cm}
\subfigure[]{
\begin{minipage}{0.24\textwidth}
    \centering
    \includegraphics[width=\linewidth]{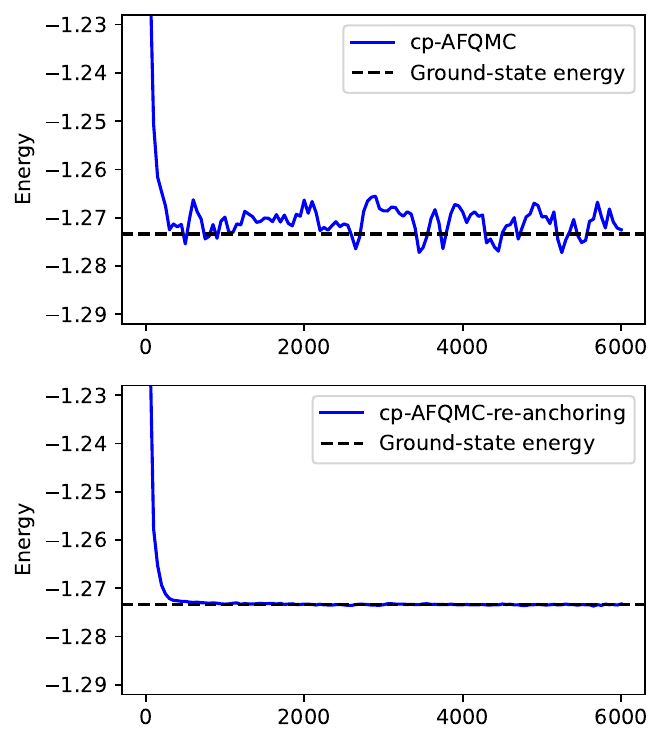}
    \vspace{-0.4cm}
\end{minipage}
}
\hspace{-0.5cm}
\subfigure[]{
\begin{minipage}{0.24\textwidth}
    \centering
    \includegraphics[width=\linewidth]{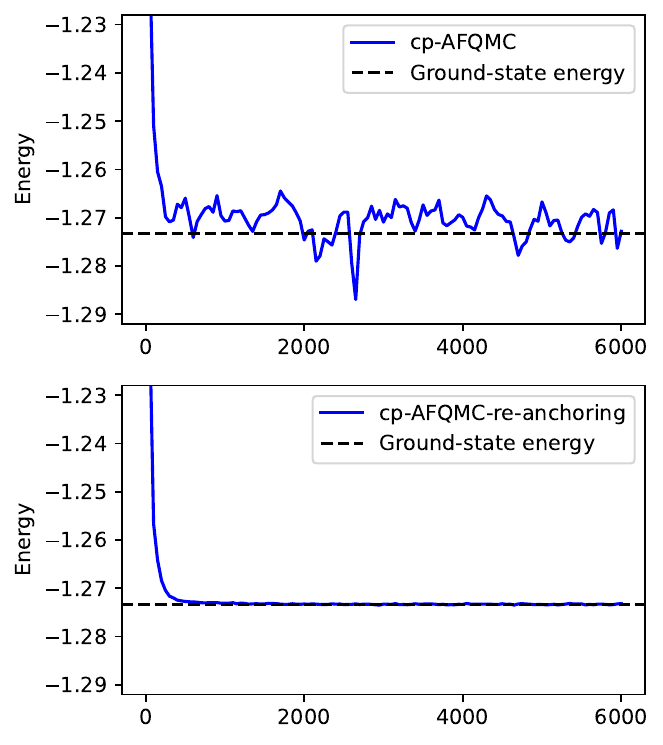}
    \vspace{-0.4cm}
\end{minipage}
}
\hspace{-0.5cm}
\subfigure[]{
\begin{minipage}{0.24\textwidth}
    \centering
    \includegraphics[width=\linewidth]{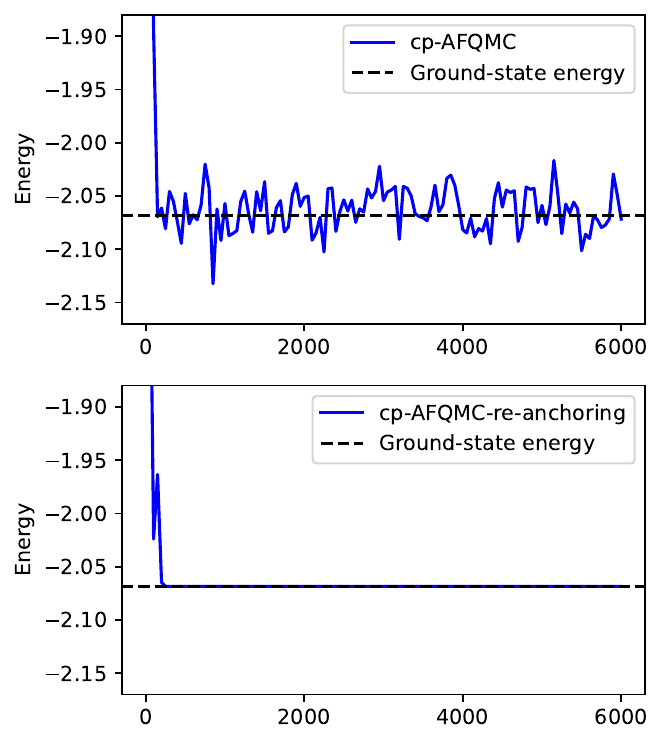}
    \vspace{-0.4cm}
\end{minipage}
}
\vspace{-0.4cm}
\caption{Comparison between the energy convergence of cp-AFQMC and cp-AFQMC-re-anchoring at $g=1$. From (a) to (d): $32$ spins in 1D; $64$ spins in 1D; $96$ spins in 1D; $4\times16$ spins in 2D (cylinder). }
\label{fig:energy}
\end{figure}

\begin{figure}[!htbp]
\centering
\subfigure[]{
\begin{minipage}{0.24\textwidth}
    \centering
    \includegraphics[width=\linewidth]{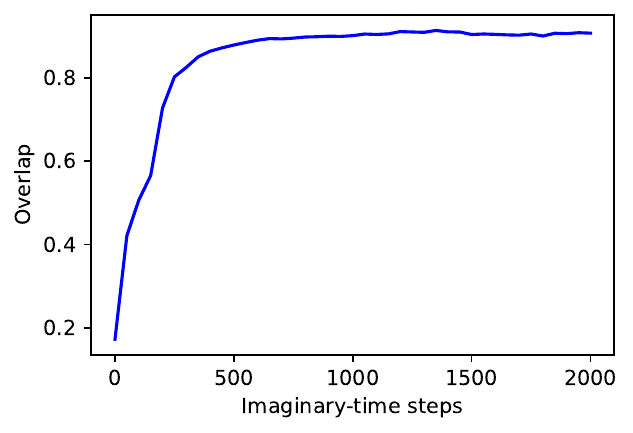}
    \vspace{-0.4cm}
\end{minipage}
}
\hspace{-0.5cm}
\subfigure[]{
\begin{minipage}{0.24\textwidth}
    \centering
    \includegraphics[width=\linewidth]{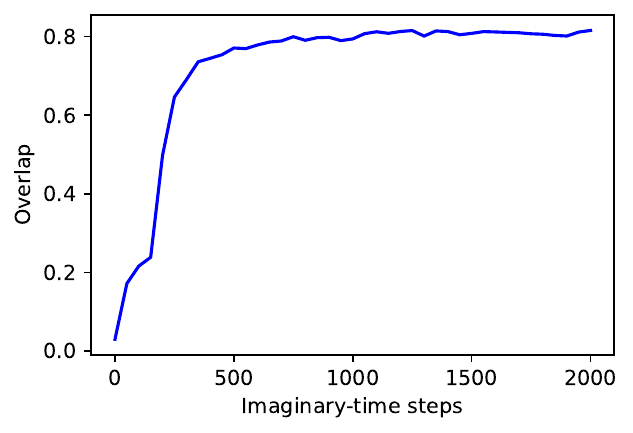}
    \vspace{-0.4cm}
\end{minipage}
}
\hspace{-0.5cm}
\subfigure[]{
\begin{minipage}{0.24\textwidth}
    \centering
    \includegraphics[width=\linewidth]{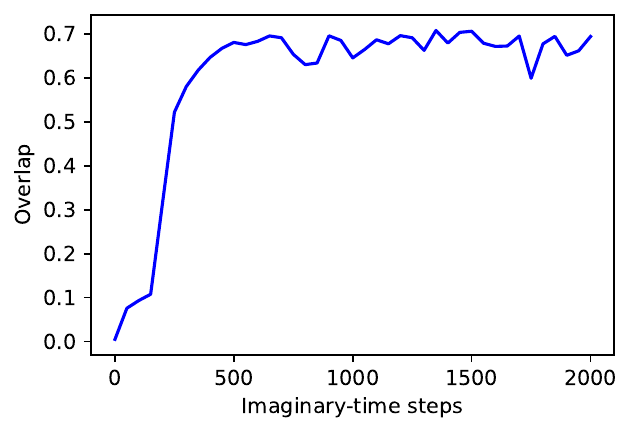}
    \vspace{-0.4cm}
\end{minipage}
}
\hspace{-0.5cm}
\subfigure[]{
\begin{minipage}{0.24\textwidth}
    \centering
    \includegraphics[width=\linewidth]{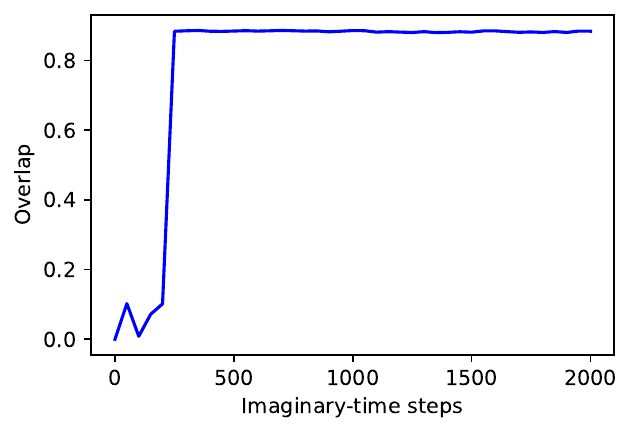}
    \vspace{-0.4cm}
\end{minipage}
}
\vspace{-0.4cm}
\caption{The overlap between the ground-state wavefunction and the sketched wavefunction from cp-AFQMC-re-anchoring, in the form of a rank-$4$ TT. From (a) to (d): $32$ spins in 1D; $64$ spins in 1D; $96$ spins in 1D; $4\times16$ spins in 2D (cylinder). }
\label{fig:overlap}
\end{figure}

\begin{table}[htbp!]
\centering
\begin{tabular}{|c|c|c|}
\hline
 & cp-AFQMC & cp-AFQMC-re-anchoring \\
\hline
$32$ spins 1D & $(+1.35\pm0.31)\times10^{-3}$ & $(+0.44\pm2.43)\times10^{-5}$ \\
\hline
$64$ spins 1D & $(+1.88\pm0.37)\times10^{-3}$ & $(+0.77\pm1.07)\times10^{-5}$ \\
\hline
$96$ spins 1D & $(+2.49\pm0.33)\times10^{-3}$ & $(-4.95\pm8.94)\times10^{-6}$ \\
\hline
$4\times16$ spins 2D (cylinder) & $(+4.18\pm1.42)\times10^{-3}$ & $(+0.50\pm1.46)\times10^{-6}$\\
\hline

\end{tabular}
\vspace{-0.25cm}

\caption{Relative error of the energy. 
The new algorithm leads to a reduction of 
statistical error by more than two orders of magnitude, or $\mathcal{O}(10,000)$ 
reduction
in computer time. Systematic errors are also reduced.
The systematic errors 
still present in 
cp-AFQMC comes from the fixed imaginary-time step and the finite population size, which can be eliminated by extrapolating. A detailed discussion is provided in the appendix. }
\label{tab:error}
\end{table}

As an example of a 2D system, we test our algorithm on the $11\times11$ lattice with open boundary at the quantum critical point $g\approx 3.044$ \cite{blote2002cluster}. In our experiments, we take $g=3$. The results are shown in Figure~\ref{fig:11*11}. 
For such a 2D system, there is no analytic solution for the ground-state energy anymore. A reference for the ground-state energy given by DMRG is $-3.17210 \pm 1\times 10^{-5}$ \cite{lubasch2014algorithms}, while the energy estimated by our method is $-3.17212 \pm 4\times 10^{-5}$. 

\begin{figure}[!htbp]
\centering
\subfigure[Vanilla cp-AFQMC]{
\begin{minipage}{0.49\textwidth}
    \centering
    \includegraphics[width=0.8\linewidth]{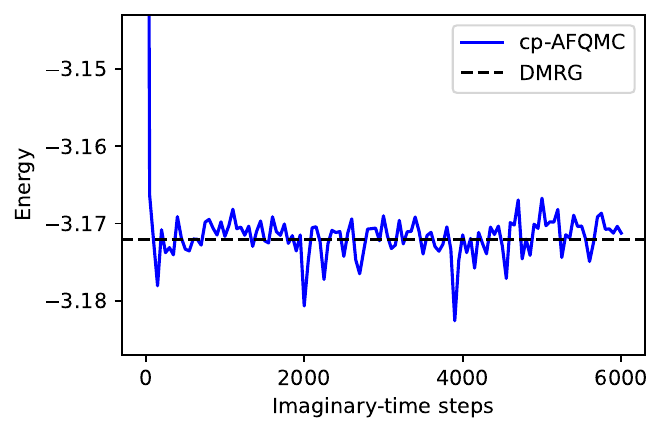}
    \vspace{0cm}
\end{minipage}
}
\hspace{-1.8cm}
\subfigure[cp-AFQMC-re-anchoring]{
\begin{minipage}{0.49\textwidth}
    \centering
    \includegraphics[width=0.8\linewidth]{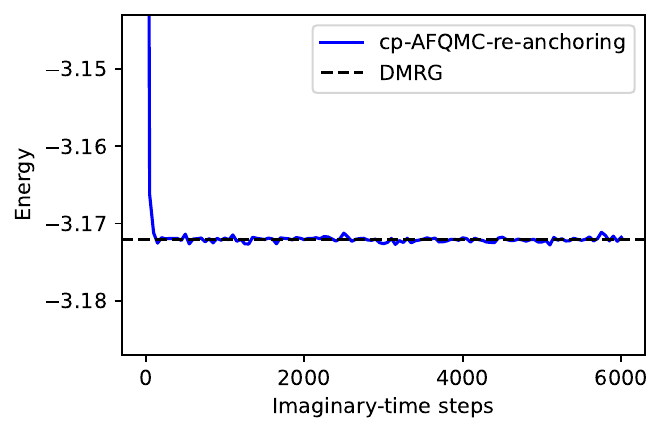}
    \vspace{0.cm}
\end{minipage}
}
\vspace{-0.4cm}
\caption{Energy convergence of (a) vanilla cp-AFQMC and (b) cp-AFQMC-re-anchoring for $11\times11$ system with open boundary at $g=3$. The reference ground-state energy, shown as the black dashed line, is obtained from DMRG \cite{lubasch2014algorithms}. }
\label{fig:11*11}
\end{figure}

\subsection{Comparison with DMRG}\label{sec:compare with DMRG}


Finally, we demonstrate that our algorithm allows the calculation of the ground-state energy to go beyond DMRG. The first comparison is conducted on the TFI system, where the spins are arranged on a $4\times4$ lattice in 2D with a periodic boundary condition. Exact diagonalization can be performed for such a system, providing a reliable reference for comparison. 
We begin by focusing on the wavefunctions obtained from DMRG and cp-AFQMC-re-anchoring at various 
transverse-field strengths $g$. The TT-rank of the wavefunctions is set to $4$. As illustrated in Figure~\ref{fig:overlap compare dmrg}, the overlap between the DMRG wavefunction and the ground state exhibits a sharp transition as the transverse field $g$ changes. The transition occurs in the interval from $2.0$ to $2.6$. For $g$ larger than the transition interval, the DMRG wavefunction closely approximates the ground state, achieving an overlap of around $0.98$. However, when $g$ is smaller than the transition interval, DMRG does not recover the ground-state wavefunction, with an overlap only around $0.7$. Our algorithm effectively resolves this issue. The sketched wavefunction from the cp-AFQMC-re-anchoring procedure consistently recovers the ground state for various $g$, maintaining an overlap above $0.96$ throughout. 

\begin{figure}[!htbp]
\centering
\subfigure[Wavefunction overlap]{
\begin{minipage}{0.49\textwidth}
    \centering
    \includegraphics[height=4cm, width=6cm]{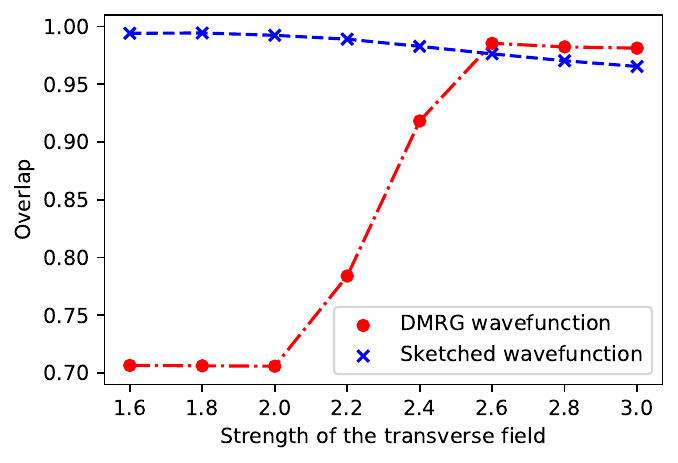}
    \vspace{0.1cm}
    \label{fig:overlap compare dmrg}
\end{minipage}
}
\hspace{-1.8cm}
\subfigure[Energy estimation error]{
\begin{minipage}{0.49\textwidth}
    \centering
    \includegraphics[height=4cm, width=6cm]{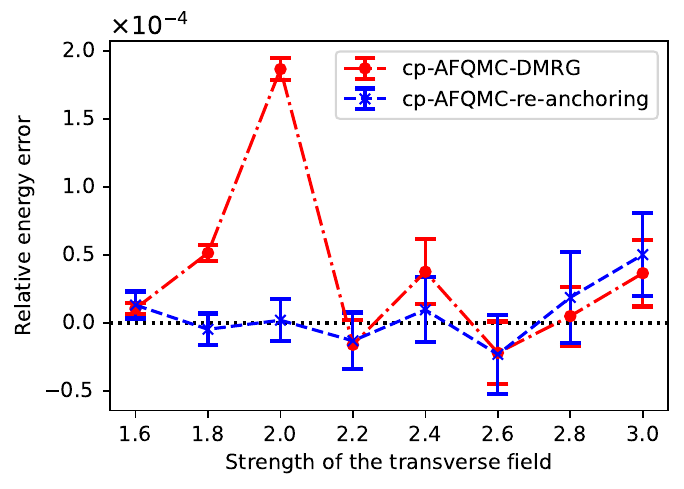}
    \vspace{0.1cm}
    \label{fig:qmc energy compare}
\end{minipage}
}
\vspace{-0.4cm}
\caption{Comparison between cp-AFQMC-re-anchoring and DMRG for the $4\times4$ system with periodic boundary with different transverse-field strengths. (a): comparison of the overlap between the ground-state wavefunction and the wavefunctions obtained by DMRG and TT-sketching; (b): comparison of the relative error in ground-state energy estimation by cp-AFQMC-DMRG and cp-AFQMC-re-anchoring. }
\end{figure}


In addition to the fact that our sketched wavefunction generally provides a better approximation of the ground state, cp-AFQMC-re-anchoring can also yield a more accurate estimation of the ground-state energy than cp-AFQMC with DMRG wavefunction as the trial function (cp-AFQMC-DMRG). 
We take a closer look at $g=2.0$, where the two algorithms exhibit a significant difference. The overlap between the ground-state and the sketched wavefunction during the cp-AFQMC-re-anchoring procedure is shown in Figure~\ref{fig:4*4 overlap}. 
In terms of the ground-state energy, Figure~\ref{fig:4*4 energy} clearly shows that cp-AFQMC-re-anchoring provides a more accurate estimation compared to cp-AFQMC-DMRG. Specifically, the relative energy error for cp-AFQMC-re-anchoring is $(0.59\pm1.57) \times 10^{-5}$, while the error for cp-AFQMC-DMRG is $(+1.87 \pm 0.08) \times 10^{-4}$, which exhibits a substantial bias.

\begin{figure}[!htbp]
\centering
\subfigure[Wavefunction overlap]{
\begin{minipage}{0.49\textwidth}
    \centering
    \includegraphics[height=4cm, width=6cm]{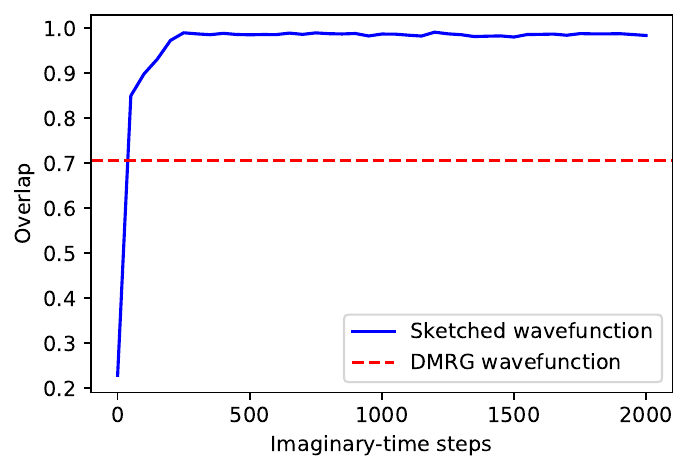}
    \label{fig:4*4 overlap}
\end{minipage}
}
\hspace{-1.5cm}
\subfigure[Energy convergence]{
\begin{minipage}{0.49\textwidth}
    \centering
    \includegraphics[height=4cm, width=6cm]{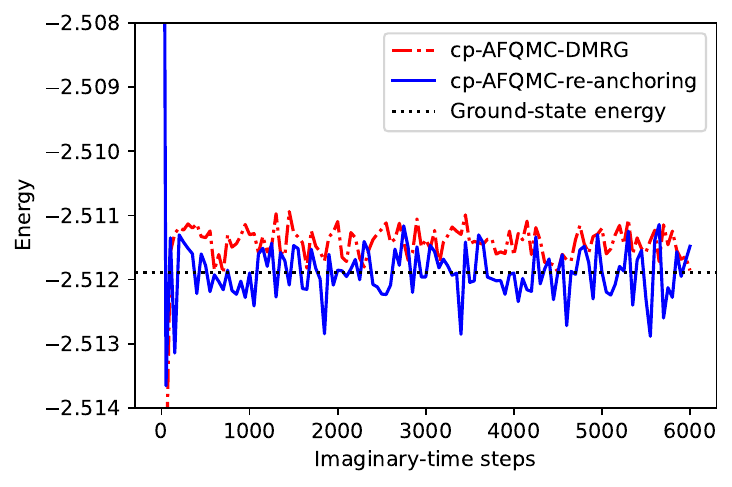}
    \label{fig:4*4 energy}
\end{minipage}
}
\vspace{-0.4cm}
\caption{Comparison between cp-AFQMC-re-anchoring and cp-AFQMC-DMRG for the $4\times4$ system with periodic boundary at $g=2.0$. (a): comparison of the overlap between the ground-state wavefunction and the wavefunctions obtained by both methods, in the form of a rank-$4$ TT; (b): comparison of the energy convergence. }
\end{figure}

Apart from ground-state energy, 
we can also assess the advantages of our algorithm
in the estimation of other observables. To estimate the expectation of an arbitrary observable $O$ in the ground state
\begin{equation*}
    \braket{O}_\mathrm{gs} = \frac{\braket{\Psi_0,O\Psi_0}}{\braket{\Psi_0,\Psi_0}}, 
\end{equation*}
we apply the extrapolated estimator \cite{whitlock1979properties}
\begin{equation*}
    \braket{O}_\mathrm{extrap} = 2\braket{O}_\mathrm{mix} - \braket{O}_\mathrm{tr}, 
\end{equation*}
where $\braket{O}_\mathrm{mix}$
denotes the mixed estimator
\begin{equation*}
    \braket{O}_\mathrm{mix} = \frac{\braket{\Psi_\mathrm{tr},O \widehat{\Psi}^{(n)}}}{\braket{\Psi_\mathrm{tr},\widehat{\Psi}^{(n)}}}, 
\end{equation*}
which is estimated from the cp-AFMQC procedure, and
\begin{equation*}
    \braket{O}_\mathrm{tr} = \frac{\braket{\Psi_\mathrm{tr},O\Psi_\mathrm{tr}}}{\braket{\Psi_\mathrm{tr},\Psi_\mathrm{tr}}}. 
\end{equation*}
Because the quality of our trial wavefunction is very high, the extrapolated estimator is sufficient to remove any residual systematic error. 
As an example, we study the average spin in $z$-direction:
\begin{equation}\label{eq:ave spin}
    \mathcal{S}_z = \frac{1}{d}\sum_{i=1}^d \sigma_i^z. 
\end{equation}
Due to the symmetry of the Hamiltonian \eqref{eq:TFI hamiltonian} under spin-flip transformations in $\sigma_z$, specifically the invariance of the Hamiltonian when all $\sigma_i^z$ are flipped to $-\sigma_i^z$, the expectation value of $\mathcal{S}_z$ in the ground state is $\braket{\mathcal{S}_z}_\mathrm{gs}=0$. Numerical results in Figure~\ref{fig:average spin} 
demonstrate that cp-AFQMC-re-anchoring preserves the $\sigma_z$ symmetry,  
which is broken in cp-AFQMC-DMRG. 

\begin{figure}[!htbp]
\centering
\includegraphics[width=0.5\linewidth]{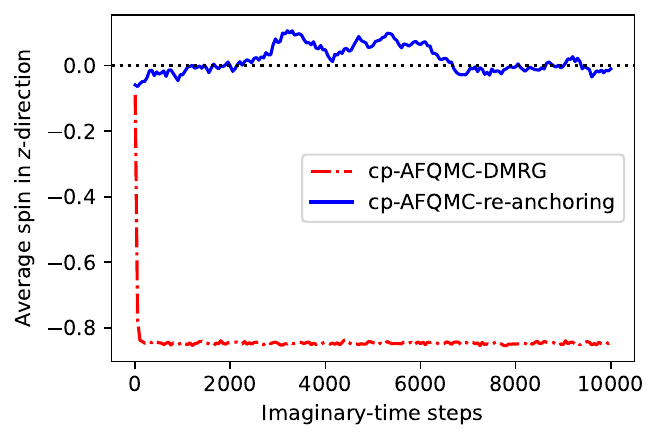}
\vspace{-0.4cm}
\caption{Average spin in $z$-direction (defined in \eqref{eq:ave spin}) of the $4\times4$ system with periodic boundary at $g=2.0$. }
\label{fig:average spin}
\end{figure}

Another example is the spin-correlation. For two spins $i$ and $j$, their correlation is calculated as
\begin{equation}\label{eq:spin correlation}
    \mathrm{Cor}(i,j) = 
    \frac{\mathrm{Cov}(\sigma_i^z,\sigma_j^z)}{\sqrt{\mathrm{Var}(\sigma_i^z) \mathrm{Var}(\sigma_j^z)}} = 
    \frac{\braket{\sigma_i^z \sigma_j^z} - \braket{\sigma_i^z} \braket{\sigma_j^z}}{\sqrt{(1-\braket{\sigma_i^z}^2) (1-\braket{\sigma_j^z}^2})}, 
\end{equation}
where we have incorporated the fact that $(\sigma_i^z)^2=(\sigma_j^z)^2=I$.

Under periodic boundary conditions, 
the correlation between two spins depends solely on their Taxicab distance \cite{sowell1989taxicab}, which is simply the number of steps needed to reach one spin from another, as illustrated in Figure~\ref{fig:spin-label}. 
We select spin $1$ as the reference (highlighted in red) and monitor its correlation with the spins labeled in green along the arrows. The comparison of our algorithm and cp-AFQMC-DMRG is presented in Figure~\ref{fig:spin-correlation by distance}. 
It is seen that
cp-AFQMC-re-anchoring recovers the correct spin-correlation perfectly, whereas cp-AFQMC-DMRG does not.


\begin{figure}[!htbp]
\centering
\subfigure[Labels of the spins]{
\begin{minipage}{0.49\textwidth}
    \centering
    \vspace{-0.25cm}
    \includegraphics[width=0.6\linewidth]{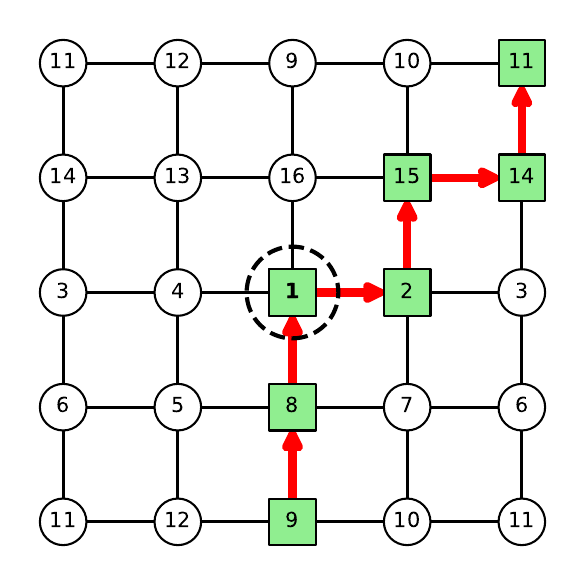}
    \label{fig:spin-label}
\end{minipage}
}
\hspace{-2.5cm}
\subfigure[Spin-correlation]{
\begin{minipage}{0.49\textwidth}
    \centering
    \includegraphics[width=0.8\linewidth]{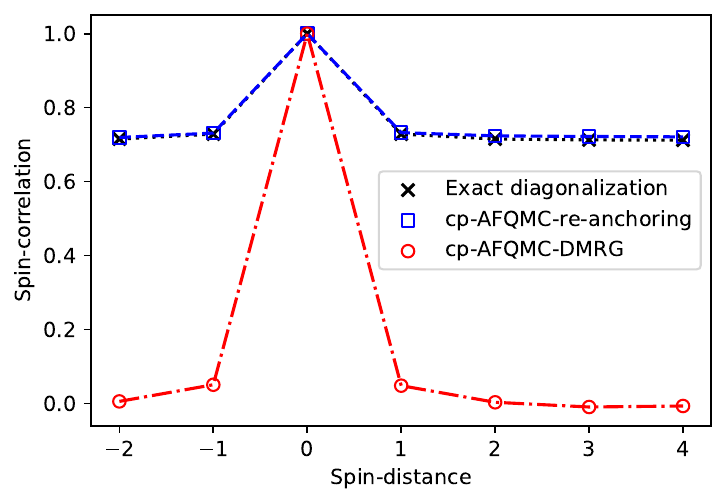}
    \label{fig:spin-correlation by distance}
\end{minipage}
}
\vspace{-0.4cm}
\caption{Spin-correlation (defined in \eqref{eq:spin correlation}) of the $4\times4$ system with periodic boundary at $g=2.0$. (a): labeling of the spins; (b): correlations with the reference spin $1$. 
}
\label{fig:spin-correlation 4*4}
\end{figure}

As a further illustration, we test both algorithms on a larger system with 64 spins arranged on an $8\times8$ lattice with periodic boundary, and estimate the ground-state energy of this system. The transverse field is set to $g=3.0$. For such a large system, exact diagonalization is no longer feasible, so the reference for the ground-state energy is obtained using DMRG with increasing ranks. As shown in Figure~\ref{fig:8*8 reanchoring energy}, the ground-state energy estimated by cp-AFQMC-re-anchoring matches the DMRG results obtained with a very high rank TT. The energy convergence plots for both cp-AFQMC-re-anchoring and cp-AFQMC-DMRG are presented in Figure~\ref{fig:8*8 qmc energy}. In both cases, the rank of the trial wavefunctions is set to $4$. 
As can be seen, the energy estimated by cp-AFQMC-DMRG exhibits a larger error compared to that of our algorithm. 

\begin{figure}[!htbp]
\centering
\subfigure[]{
\begin{minipage}{0.49\textwidth}
    \centering
    \includegraphics[width=0.8\linewidth]{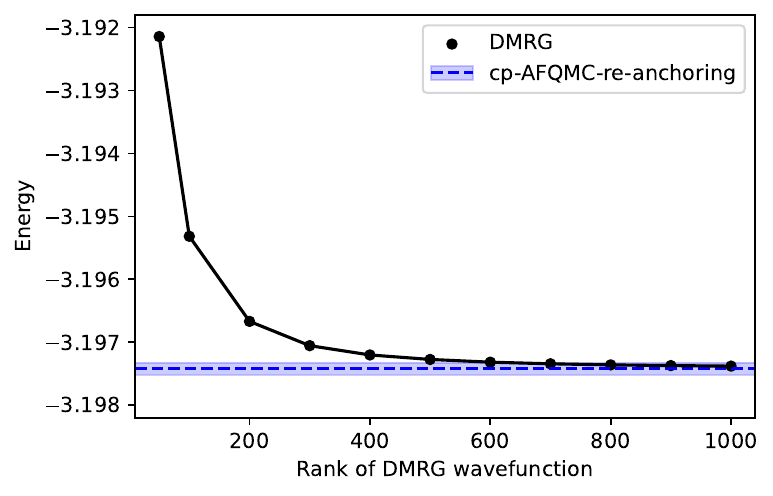}
    \label{fig:8*8 reanchoring energy}
\end{minipage}
}
\hspace{-1.8cm}
\subfigure[]{
\begin{minipage}{0.49\textwidth}
    \centering
    \includegraphics[width=0.8\linewidth]{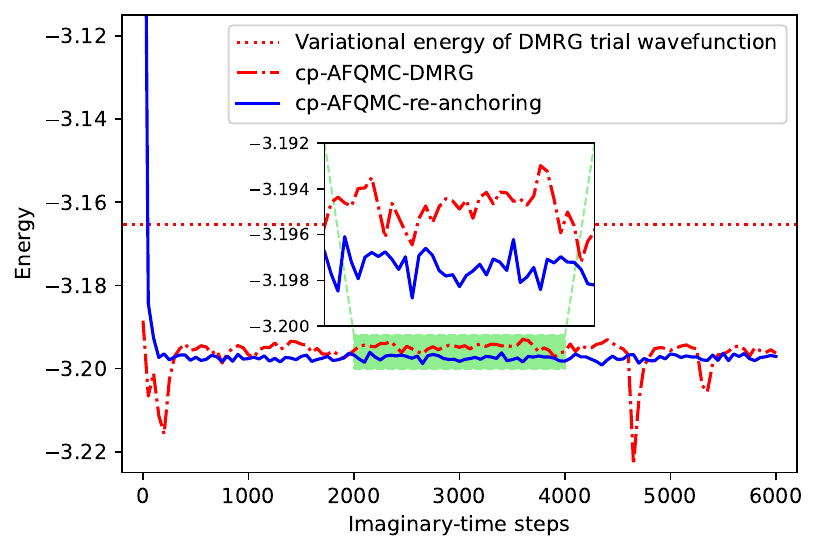}
    \label{fig:8*8 qmc energy}
\end{minipage}
}
\vspace{-0.4cm}
\caption{Results for the $8\times8$ system with periodic boundary at $g=3.0$. The rank of the trial wavefunctions in cp-AFQMC is set to $4$. (a): comparison between the ground-state energy estimated by DMRG with increasing ranks and cp-AFQMC-re-anchoring; (b): comparison between the energy convergence of cp-AFQMC-re-anchoring and cp-AFQMC-DMRG. }
\end{figure}

\section{Conclusion}\label{sec:conclusion}
In this paper, we proposed a framework for studying the ground-state properties of quantum many-body systems. Our algorithm improves cp-AFQMC by periodically updating the trial wavefunction using TT-sketching. 
Systematic analysis of the algorithm 
and detailed numerical tests 
on the transverse-field Ising model are presented.
The algorithm combines key features of AFQMC with TT/MPS, providing a route for a more powerful approach that extends the capabilities of each. 
This work opens a number of future directions, including extension and application of
this framework to electronic systems. 

\section*{Acknowledgments}
YK is partially funded by DMS-2339439, DE-SC0022232, and Sloan Research Fellowship. The Flatiron
Institute is a division of the Simons Foundation.

The authors would like to thank Miles Stoudenmire for helpful discussions on the use of the ITensor library in our numerical simulations.

\appendix

\section{Time Complexity of TT Calculation}

In this section, we investigate numerically the time complexity of tensor-train calculations. Specifically, we verify that computing the inner product between a TT trial wavefunction and the walkers scales linearly with both the number of spins and the number of walkers, as shown in Figure~\ref{fig:tt overlap scaling}. Similarly, the computational cost of TT-sketching has a linear scaling in both parameters, as illustrated in Figure~\ref{fig:tt-sketching scaling}.

\begin{figure}[!htbp]
\centering
\subfigure[Scaling with number of spins]{
\begin{minipage}{0.49\textwidth}
    \centering
    \includegraphics[width=0.8\linewidth]{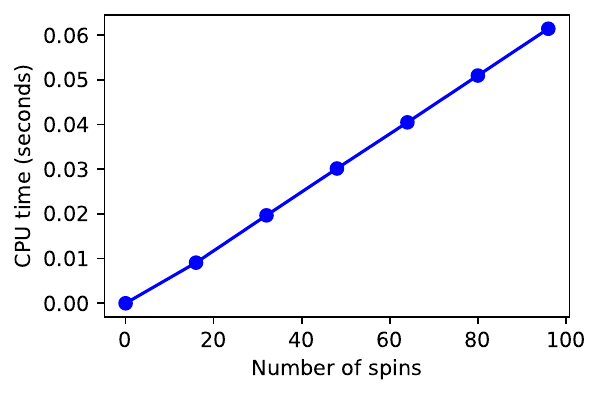}
    \vspace{0.1cm}
\end{minipage}
}
\hspace{-1.5cm}
\subfigure[Scaling with number of walkers]{
\begin{minipage}{0.49\textwidth}
    \centering
    \includegraphics[width=0.8\linewidth]{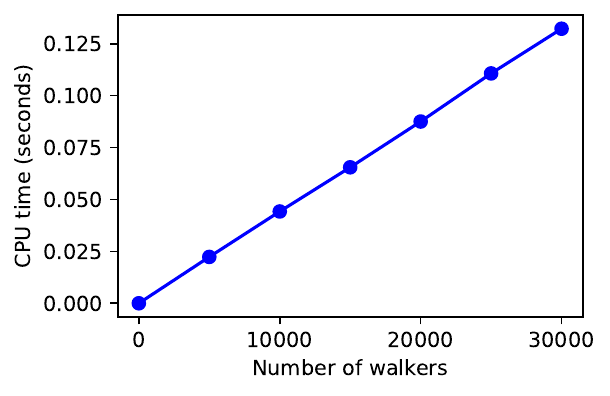}
    \vspace{0.1cm}
\end{minipage}
}
\vspace{-0.4cm}
\caption{Computational cost of calculating the inner product of a TT trial wavefunction and the walkers. (a): different number of spins with $2000$ walkers; (b): different number of walkers for $16$ spins. 
}
\label{fig:tt overlap scaling}
\end{figure}

\begin{figure}[!htbp]
\centering
\subfigure[Scaling with number of spins]{
\begin{minipage}{0.49\textwidth}
    \centering
    \includegraphics[width=0.8\linewidth]{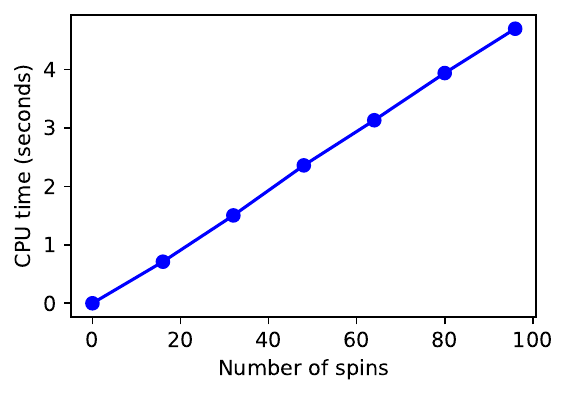}
    \vspace{0.1cm}
\end{minipage}
}
\hspace{-1.5cm}
\subfigure[Scaling with number of walkers]{
\begin{minipage}{0.49\textwidth}
    \centering
    \includegraphics[width=0.8\linewidth]{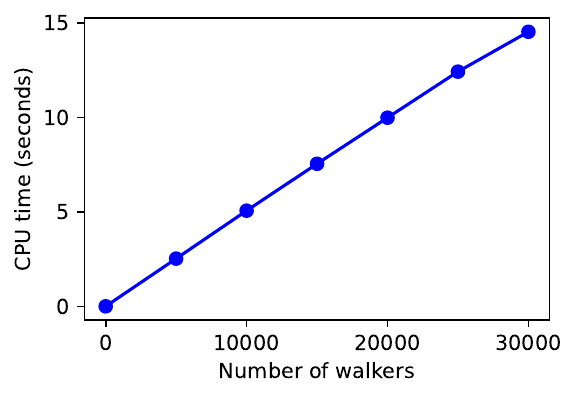}
    \vspace{0.1cm}
\end{minipage}
}
\vspace{-0.4cm}
\caption{Computational cost of TT-sketching. (a): different number of spins with $2000$ walkers; (b): different number of walkers for $16$ spins. }
\label{fig:tt-sketching scaling}
\end{figure}

\section{Details on Convergence Analysis of cp-AFQMC}\label{sec:thm proof}

In this section, we provide proofs for Lemma~\ref{lemma:energy error decreasing new}, Theorem~\ref{thm:energy convergence} and Theorem~\ref{thm: walker convergence}. 

\begin{proof}[Proof of Lemma~\ref{lemma:energy error decreasing new}]
The energy can be written explicitly as
\begin{equation*}
    \frac{\braket{\Psi_\mathrm{tr},H\Phi_k^{(n)}}}{\braket{\Psi_\mathrm{tr},\Phi_k^{(n)}}} = \frac{\sum_{j=0}^{2^d-1}c_jc_j^\prime E_j}{\sum_{j=0}^{2^d-1}c_jc_j^\prime}, 
\end{equation*}
and thus
\begin{equation}\label{eq:energy error old}
    \frac{\braket{\Psi_\mathrm{tr},H\Phi_k^{(n)}}}{\braket{\Psi_\mathrm{tr},\Phi_k^{(n)}}} - E_0 = \frac{\sum_{j=1}^{2^d-1}c_jc_j^\prime (E_j-E_0)}{\sum_{j=0}^{2^d-1}c_jc_j^\prime}. 
\end{equation}
Similarly, we have
\begin{equation}\label{eq:energy error new}
    \frac{\braket{\Psi_\mathrm{tr},He^{-\Delta\tau H}\Phi_k^{(n)}}}{\braket{\Psi_\mathrm{tr},e^{-\Delta\tau H}\Phi_k^{(n)}}} - E_0 = \frac{\sum_{j=1}^{2^d-1}c_jc_j^\prime e^{-\Delta\tau E_j} (E_j-E_0)}{\sum_{j=0}^{2^d-1}c_jc_j^\prime e^{-\Delta\tau E_j}}, 
\end{equation}
since \eqref{eq:phi_k^n expansion} directly gives the expansion
\begin{equation*}
    e^{-\Delta\tau H} \Phi_k^{(n)} =  c_0^\prime e^{-\Delta\tau E_0}\Psi_0 +  c_1^\prime e^{-\Delta\tau E_1}\Psi_1 + \cdots +  c_{2^d-1}^\prime e^{-\Delta\tau E_{2^d-1}}\Psi_{2^d-1}. 
\end{equation*}

Now, we compare \eqref{eq:energy error old} and \eqref{eq:energy error new}:
\begin{equation}\label{eq:energy error compare}
    \frac{\abs{\frac{\braket{\Psi_\mathrm{tr},He^{-\Delta\tau H}\Phi_k^{(n)}}}{\braket{\Psi_\mathrm{tr},e^{-\Delta\tau H}\Phi_k^{(n)}}} - E_0}}{\abs{\frac{\braket{\Psi_\mathrm{tr},H\Phi_k^{(n)}}}{\braket{\Psi_\mathrm{tr},\Phi_k^{(n)}}} - E_0}}
    = \frac{\abs{\sum_{j=1}^{2^d-1}c_jc_j^\prime e^{-\Delta\tau E_j} (E_j-E_0)}}{\abs{\sum_{j=1}^{2^d-1}c_jc_j^\prime (E_j-E_0)}}
    \cdot
    \frac{\abs{\sum_{j=0}^{2^d-1}c_jc_j^\prime}}{\abs{\sum_{j=0}^{2^d-1}c_jc_j^\prime e^{-\Delta\tau E_j}}}. 
\end{equation}
For the first term of the RHS of \eqref{eq:energy error compare}, the assumption \eqref{eq:assmption c_1 dominates} gives
\begin{equation}\label{eq:first term of energy compare}
\begin{split}
    \frac{\abs{\sum_{j=1}^{2^d-1}c_jc_j^\prime e^{-\Delta\tau E_j} (E_j-E_0)}}{\abs{\sum_{j=1}^{2^d-1}c_jc_j^\prime (E_j-E_0)}} \le &
    \frac{\abs{c_1c_1^\prime e^{-\Delta\tau E_1}(E_1-E_0)} + \abs{\sum_{j=2}^{2^d-1}c_jc_j^\prime e^{-\Delta\tau E_j} (E_j-E_0)}}{\abs{c_1c_1^\prime} (E_1-E_0) - \abs{\sum_{j=2}^{2^d-1}c_jc_j^\prime (E_j-E_0)}} \\
    \le & \frac{e^{-\Delta\tau E_1} \abs{c_1c_1^\prime} (E_1-E_0) + a_1 e^{-\Delta\tau E_1 }\vert c_1\vert (E_1-E_0)\sqrt{\sum_{j=2}^{2^d-1}{c_j^\prime}^2}}{\abs{c_1c_1^\prime} (E_1-E_0) - \sum_{j=2}^{2^d-1}\abs{c_jc_j^\prime} (E_j-E_0)} \\
    \le & \frac{e^{-\Delta\tau E_1} \abs{c_1c_1^\prime} (E_1-E_0) + a_1 e^{-\Delta\tau E_1 }\vert c_1\vert (E_1-E_0)\sqrt{\sum_{j=2}^{2^d-1}{c_j^\prime}^2}}{\abs{c_1c_1^\prime} (E_1-E_0) - a_1 \vert c_1\vert (E_1-E_0)\sqrt{\sum_{j=2}^{2^d-1}{c_j^\prime}^2}} \\
    \le & e^{-\Delta\tau E_1}\frac{1+a_1\frac{\sqrt{\sum_{j=2}^{2^d-1}{c_j^\prime}^2}}{\vert c'_1\vert}}{1-a_1\frac{\sqrt{\sum_{j=2}^{2^d-1}{c_j^\prime}^2}}{\vert c'_1\vert}}. 
\end{split}
\end{equation}
For the second term of the RHS of \eqref{eq:energy error compare}, we have
\begin{equation}\label{eq:second term of energy compare}
\begin{split}
    \frac{\abs{\sum_{j=0}^{2^d-1}c_jc_j^\prime}}{\abs{\sum_{j=0}^{2^d-1}c_jc_j^\prime e^{-\Delta\tau E_j}}} 
    \le & \frac{\abs{c_0c_0^\prime} + \sum_{j=1}^{2^d-1}\abs{c_jc_j^\prime}}{e^{-\Delta\tau E_0} \abs{c_0c_0^\prime } - \sum_{j=1}^{2^d-1}\abs{c_jc_j^\prime} e^{-\Delta\tau E_j}} \\
    \le & \frac{\abs{c_0c_0^\prime} + \sum_{j=1}^{2^d-1}\abs{c_jc_j^\prime}}{e^{-\Delta\tau E_0} \abs{c_0c_0^\prime } - a_0 e^{-\Delta\tau E_0} \abs{c_0}\sqrt{\sum_{j=2}^{2^d-1}{c_j^\prime}^2}} \\
    \le & e^{\Delta\tau E_0}\frac{1+a_0\frac{\sqrt{\sum_{j=1}^{2^d-1}{c_j^\prime}^2}}{\vert c'_0\vert}}{1-a_0\frac{\sqrt{\sum_{j=1}^{2^d-1}{c_j^\prime}^2}}{\vert c'_0\vert}}.
\end{split}
\end{equation}
Plugging \eqref{eq:first term of energy compare} and \eqref{eq:second term of energy compare} back into \eqref{eq:energy error compare} then gives \eqref{eq:lemma energy error decrease in exact power iteration}.


\end{proof}

\begin{proof}[Proof of Theorem~\ref{thm:energy convergence}]

We split the energy into a deterministic part and a fluctuation part:
\begin{equation}\label{eq:theorem energy convergence split}
    \abs{\frac{\braket{\Psi_\mathrm{tr},H\widetilde{B}_k^{(n)}(\vec{x})\Phi_k^{(n)}}}{\braket{\Psi_\mathrm{tr},\widetilde{B}_k^{(n)}(\vec{x})\Phi_k^{(n)}}} -E_0}\leq 
    \abs{\frac{\braket{\Psi_\mathrm{tr},He^{-\Delta\tau H}\Phi_k^{(n)}}}{\braket{\Psi_\mathrm{tr},e^{-\Delta\tau H}\Phi_k^{(n)}}} - E_0}+ \abs{\frac{\braket{\Psi_\mathrm{tr},H(\widetilde{B}_k^{(n)}(\vec{x})-e^{-\Delta\tau H})\Phi_k^{(n)}}}{\braket{\Psi_\mathrm{tr},e^{-\Delta\tau H}\Phi_k^{(n)}}}}, 
\end{equation}
where we have used \eqref{eq:lemma: no sign problem zero variance} for the denominator. 
The bound on the first term directly follows from Lemma~\ref{lemma:energy error decreasing new}. For the second term,  plugging in the decomposition \eqref{eq:trial decompose} gives
\begin{eqnarray*}\label{eq:theorem energy convergence one step fluctuation part}
     & & \frac{\braket{\Psi_\mathrm{tr},H(\widetilde{B}_k^{(n)}(\vec{x})-e^{-\Delta\tau H})\Phi_k^{(n)}}}{\braket{\Psi_\mathrm{tr},e^{-\Delta\tau H}\Phi_k^{(n)}}} \cr
    &=& \frac{\braket{H\Psi_\mathrm{tr},(\widetilde{B}_k^{(n)}(\vec{x})-e^{-\Delta\tau H})\Phi_k^{(n)}}}{\braket{\Psi_\mathrm{tr},e^{-\Delta\tau H}\Phi_k^{(n)}}} \cr
    &=& \frac{\braket{ c_0E_0\Psi_0+c_\perp H\Psi_\perp,(\widetilde{B}_k^{(n)}(\vec{x})-e^{-\Delta\tau H})\Phi_k^{(n)}}}{\braket{\Psi_\mathrm{tr},e^{-\Delta\tau H}\Phi_k^{(n)}}}\cr 
    &=& \frac{\braket{c_\perp(H-E_0I)\Psi_\perp,(\widetilde{B}_k^{(n)}(\vec{x})-e^{-\Delta\tau H})\Phi_k^{(n)}}}{\braket{\Psi_\mathrm{tr},e^{-\Delta\tau H}\Phi_k^{(n)}}} \cr
    &=& \frac{\braket{c_\perp(H-E_0I)\Psi_\perp,(\widetilde{B}_k^{(n)}(\vec{x})-e^{-\Delta\tau H})\Phi_k^{(n)}}}{\braket{\Psi_\mathrm{tr},e^{-\Delta\tau H}\Phi_k^{(n)}}} \cr
    &\le& c_\perp\|H-E_0 I \|_2\|\widetilde{B}_k^{(n)}(\vec{x})-e^{-\Delta\tau H} \|_2\frac{\| \mathcal{P}_\perp \Phi_k^{(n)} \|}{\vert\braket{\Psi_\mathrm{tr},e^{-\Delta\tau H}\Phi_k^{(n)}}\vert}\cr
    &=&  c_\perp\|H-E_0 I \|_2\|\widetilde{B}_k^{(n)}(\vec{x})-e^{-\Delta\tau H} \|_2\frac{\| \mathcal{P}_\perp \Phi_k^{(n)} \|}{\vert c_0\braket{\Psi_0,e^{-\Delta\tau H}\Phi_k^{(n)}}+c_\perp\braket{\Psi_\perp,e^{-\Delta\tau H}\Phi_k^{(n)}}\vert}\cr
    &\leq &  \|H-E_0 I \|_2\frac{\|\widetilde{B}_k^{(n)}(\vec{x})-e^{-\Delta\tau H} \|_2}{e^{-\Delta\tau E_0}}\left\vert \frac{c_\perp}{c_0}\frac{\| \mathcal{P}_\perp \Phi_k^{(n)} \|}{\braket{\Psi_0,\Phi_k^{(n)}}}\right\vert\bigg/\left(1+a_0 e^{-\Delta(E_1-E_0)}\frac{\| \mathcal{P}_\perp \Phi_k^{(n)} \|}{\vert \braket{\Psi_0,\Phi_k^{(n)}}\vert}\right)\cr
    &=& \frac{\|\widetilde{B}_k^{(n)}(\vec{x})-e^{-\Delta\tau H} \|_2}{e^{-\Delta\tau E_0}}  \frac{a_0\tan\theta_k^{(n)}\|H-E_0 I \|_2}{1+a_0 e^{-\Delta(E_1-E_0)}\tan\theta_k^{(n)}},
\end{eqnarray*}
where the second equality follows from \eqref{eq:assumption c_0 dominates} and Cauchy–Schwarz inequality, and we use $\mathcal{P}_\perp$ to denote the orthogonal projector onto the space spanned by $\{\Psi_1,\ldots \Psi_{2^d-1}\}$. Plugging this and the bound in Lemma~\ref{lemma:energy error decreasing new} into \eqref{eq:theorem energy convergence split} then gives \eqref{eq:theorem energy convergence}. 
\end{proof}

\begin{proof}[Proof of Theorem~\ref{thm: walker convergence}]
To estimate $\tan\theta_k^{(n+1)}(\vec{x})$, we need to calculate each component of $\widetilde{B}_k^{(n)}(\vec{x})\Phi_k^{(n)}$ under the basis $\{\Psi_j\}_j$. We first study the overlap of $\widetilde{B}_k^{(n)}(\vec{x})\Phi_k^{(n)}$ with the ground state $\Psi_0$
\begin{multline}\label{eq:psi0 component}
    \braket{\Psi_0,\widetilde{B}_k^{(n)}(\vec{x})\Phi_k^{(n)}} = \\ \left\langle\frac{\Psi_\mathrm{tr}-c_\perp \Psi_\perp}{c_0},\widetilde{B}_k^{(n)}(\vec{x})\Phi_k^{(n)}\right\rangle
    = \frac{1}{c_0}\langle\Psi_\mathrm{tr},\widetilde{B}_k^{(n)}(\vec{x})\Phi_k^{(n)}\rangle - \frac{c_\perp}{c_0} \braket{\Psi_\perp,\widetilde{B}_k^{(n)}(\vec{x})\Phi_k^{(n)}}. 
\end{multline}

For the first term, we have
\begin{equation*}
\begin{split}
    &\quad \braket{\Psi_\mathrm{tr},\widetilde{B}_k^{(n)}(\vec{x})\Phi_k^{(n)}}\\
    &=\braket{\Psi_\mathrm{tr},e^{-\Delta\tau H}\Phi_k^{(n)}} \\
    &= \braket{c_0 \Psi_0 + c_\perp \Psi_\perp,e^{-\Delta\tau H}\Phi_k^{(n)}} \\
    &= c_0 \braket{e^{-\Delta\tau E_0}\Psi_0,\Phi_k^{(n)}} + c_\perp\braket{\Psi_\perp,e^{-\Delta\tau H}\Phi_k^{(n)}}, \\
\end{split}
\end{equation*}
where in the first equality we use Lemma~\ref{lemma: no sign problem zero variance}, in the second equality we use \eqref{eq:trial decompose}, and in the third equality we use the fact that $e^{-\Delta\tau H}$ is a Hermitian matrix. Plugging this back into \eqref{eq:psi0 component} we get
\begin{equation*}
\begin{split}
\braket{\Psi_0,\widetilde{B}_k^{(n)}(\vec{x})\Phi_k^{(n)}} = e^{-\Delta\tau E_0}\braket{\Psi_0,\Phi_k^{(n)}} - \frac{c_\perp}{c_0} \braket{\Psi_\perp, E^{(n)}_k}. 
\end{split}
\end{equation*}

Now, the component of $\widetilde{B}_k^{(n)}(\vec{x})\Phi_k^{(n)}$ in  the excited subspace can be characterized as 
\begin{equation*}
\begin{split}
    \|\mathcal{P}_\perp \widetilde{B}_k^{(n)}\Phi_k^{(n)} \|&\leq \|\mathcal{P}_\perp e^{-\Delta\tau H}\Phi_k^{(n)}\| + \|\mathcal{P}_\perp \widetilde{B}_k^{(n)}(\vec{x}) - e^{-\Delta\tau H}\|\leq e^{-\Delta \tau E_1}\|\mathcal{P}_\perp \Phi_k^{(n)}\| +  \|\widetilde{B}_k^{(n)}(\vec{x}) - e^{-\Delta\tau H}\|, \\
\end{split} 
\end{equation*}
where $\mathcal{P}_\perp$ projects to 
\begin{equation*}
\begin{split}
    \tan\theta_k^{(n+1)}(\vec{x}) &=  \frac{\|\mathcal{P}_\perp \widetilde{B}_k^{(n)}\Phi_k^{(n)} \|}{\vert\braket{\Psi_0,\widetilde{B}_k^{(n)}(\vec{x})\Phi_k^{(n)}}\vert}\\
    &\leq  \frac{e^{-\Delta \tau E_1}\|\mathcal{P}_\perp \Phi_k^{(n)}\| +   \|\widetilde{B}_k^{(n)}(\vec{x}) - e^{-\Delta\tau H}\|}{ e^{-\Delta\tau E_0}\vert\braket{\Psi_0,\Phi_k^{(n)}}\vert - \abs{\frac{c_\perp}{c_0}} \vert\braket{\Psi_\perp, \widetilde{B}_k^{(n)}(\vec{x}) - e^{-\Delta\tau H}}\vert}\\
    &\leq e^{-\Delta \tau (E_1-E_0)}\frac{\|\mathcal{P}_\perp \Phi_k^{(n)}\|}{\vert\braket{\Psi_0,\Phi_k^{(n)}}\vert}\left(1+ e^{\Delta\tau E_0}\abs{\frac{c_\perp}{c_0}} O\left(\frac{ \|\widetilde{B}_k^{(n)}(\vec{x}) - e^{-\Delta\tau H}\|}{\vert\braket{\Psi_0,\Phi_k^{(n)}}\vert}\right)\right) \\
    &\quad + \frac{ \|\widetilde{B}_k^{(n)}(\vec{x}) - e^{-\Delta\tau H}\|}{e^{-\Delta \tau E_0}\braket{\Psi_0,\Phi_k^{(n)}}}\left(1+ e^{\Delta\tau E_0}\abs{\frac{c_\perp}{c_0}} O\left(\frac{ \|\widetilde{B}_k^{(n)}(\vec{x}) - e^{-\Delta\tau H}\|}{\vert\braket{\Psi_0,\Phi_k^{(n)}}\vert}\right)\right)\\
    &\leq  e^{-\Delta \tau (E_1-E_0)}\tan \theta_k^{(n)} + e^{\Delta \tau E_0}\left(\tan \theta_k^{(n)} \abs{\frac{c_\perp}{c_0}} +1\right)O\left( \frac{ \|\widetilde{B}_k^{(n)}(\vec{x}) - e^{-\Delta\tau H}\|}{\vert\braket{\Psi_0,\Phi_k^{(n)}}\vert}  \right)\\
    &\quad + e^{2\Delta \tau E_0}\abs{\frac{c_\perp}{c_0}} O\left( \frac{ \|\widetilde{B}_k^{(n)}(\vec{x}) - e^{-\Delta\tau H}\|}{\vert\braket{\Psi_0,\Phi_k^{(n)}}\vert}  \right)^2. 
\end{split}
\end{equation*}

\end{proof}

\section{Implementation of cp-AFQMC on Transverse-Field Ising Model}\label{sec:cp-afqmc detail on tfi}
In this section, we introduce the detailed implementation of Algorithm~\ref{alg:cp-AFQMC}, and how we apply it on the transverse-field Ising model. 

\subsection{Implementation details of cp-AFQMC}
To implement Algorithm~\ref{alg:cp-AFQMC}, we rewrite the wavefunction representation \eqref{eq:wavefunction representation} as
\begin{equation}\label{eq:cp walker}
    \widehat{\Psi}^{(n)} = \sum_k w_k^{(n)} \frac{\phi_k^{(n)}}{\braket{\Psi_\mathrm{tr}, \phi_k^{(n)}}}, 
\end{equation}
where $w_k^{(n)}$ denotes the weight of the $k$-th walker. In actual implementations, a walker is fully specified by
$\phi_k^{(n)}$, its associated weight $w_k^{(n)}$, and the overlap
$\langle \Psi_{\mathrm{tr}}, \phi_k^{(n)} \rangle$, which are stored and updated during the simulation. In particular, whenever $\phi_k^{(n)}$ is updated, the overlap $\langle \Psi_{\mathrm{tr}}, \phi_k^{(n)} \rangle$ is updated accordingly. 
Since the random walkers are guaranteed to always have a positive overlap with the trial wavefunction in cp-AFQMC, the weight of a valid walker is always positive $w_k^{(n)} > 0$. 
The propagation of the walkers \eqref{eq:ensemble propagation math} can then be rewritten as
\begin{eqnarray*}
e^{-\Delta\tau H}\widehat \Psi^{(n)} &\approx& \sum_k \sum_{\vec{x}:Q^{(n)}_k(\vec{x})>0}Q^{(n)}_k(\vec{x}) w^{(n)}_k\frac{P(\vec{x})}{Q^{(n)}_k(\vec{x})} \frac{B(\vec{x}){\phi}_k^{(n)}}{\braket{\Psi_\mathrm{tr}, \phi_k^{(n)}}}\cr 
&\approx& \sum_k  w^{(n)}_k \frac{P(\vec{x}^{(n)}_k)}{Q^{(n)}_k(\vec{x}^{(n)}_k)} \frac{\phi_k^{(n+1)} }{\braket{\Psi_\mathrm{tr}, \phi_k^{(n)}}}\cr 
&=& \sum_k  \left(w^{(n)}_k \frac{P(\vec{x}^{(n)}_k)}{Q^{(n)}_k(\vec{x}^{(n)}_k)} \frac{\braket{\Psi_\mathrm{tr}, \phi_k^{(n+1)}}}{\braket{\Psi_\mathrm{tr}, \phi_k^{(n)}}}\right)\frac{\phi_k^{(n+1)} }{\braket{\Psi_\mathrm{tr}, \phi_k^{(n+1)}}}\cr 
&=& \sum_k w_k^{(n+1)} \frac{\phi_k^{(n+1)} }{\braket{\Psi_\mathrm{tr}, \phi_k^{(n+1)}}},
\end{eqnarray*}
where
\begin{equation*}
 \phi_k^{(n+1)} = B(\vec{x}^{(n)}_k){\phi}_k^{(n)}, \vec{x}^{(n)}_k\sim Q^{(n)}_k,
\end{equation*}
and
\begin{equation}\label{eq:weight update factor}
w^{(n+1)}_k = w^{(n)}_k\mathcal{N}_k^{(n)}.
\end{equation}
Here, we remark that the new weight $w_k^{(n+1)}$ as defined in \eqref{eq:weight update factor} does not depend on the auxiliary field $\vec{x}$ being sampled, since the normalization factor $\mathcal{N}_k^{(n)}$ is independent of $\vec{x}$. 

With the wave function representation \eqref{eq:cp walker}, we can also estimate the ground-state energy in a straightforward way. Specifically, we use the mixed estimator
\begin{equation}\label{eq:mixed energy estimator}
    E_\text{mixed} = \frac{\braket{\Psi_\mathrm{tr},H\widehat \Psi^{(n)}}}{\braket{\Psi_\mathrm{tr},\widehat \Psi^{(n)}}} = \frac{\sum_k w^{(n)}_k E_\text{local}[\Psi_\mathrm{tr},\phi_k^{(n)}]}{\sum_k w^{(n)}_k}, 
\end{equation}
where $E_\text{local}$ represents the local energy of a walker:
\begin{equation*}
    E_\text{local}[\Psi_\mathrm{tr},\phi] = \frac{\braket{\Psi_\mathrm{tr},H\phi}}{\braket{\Psi_\mathrm{tr},\phi}}. 
\end{equation*}
As the random walk goes on, the weights of the walkers can become highly uneven: some grow very large while others shrink to negligible values. Consequently, only a small fraction of walkers effectively contribute to the energy estimation, leading to increased variance and reduced sampling efficiency. To solve this issue, we periodically apply a population control procedure, in which walkers with very small weights are removed, while those with large weights are branched into multiple copies. The detailed implementation is summarized in Algorithm~\ref{alg:ppl control}. 

\begin{algorithm}[!htbp]
\caption{Population control}\label{alg:ppl control}
\begin{algorithmic}
\Require walkers $\phi_k$, weights $w_k$, and corresponding $\braket{\Psi_\mathrm{tr}, \phi_k}$. 
\Ensure adjusted walkers $\phi_k$, weights $w_k$, and corresponding $\braket{\Psi_\mathrm{tr}, \phi_k}$. 

\State normalize the weights so that the average value of the weights is $1$. 

\State $s\leftarrow \text{a uniformly random number in }(0,1)$,
\State $n_1\leftarrow1$, 

\For{$k=1,2,\cdots,N$}
\State $s\leftarrow s + w_k$; 
\State $n_2\leftarrow \text{the smallest integer that is greater than or equal to } s$; 
\State set the walkers $\phi_{n_1},\cdots,\phi_{n_2-1}$ to be $\phi_k$, and corresponding overlaps to be $\braket{\Psi_\mathrm{tr}, \phi_k}$; 
\State $n_1\leftarrow n_2$
\EndFor
\State reset all the weights to be $w_k=1$. 

\end{algorithmic}
\end{algorithm}

\subsection{cp-AFQMC on TFI model}\label{sec:cpmc on tfi}

To apply the cp-AFQMC procedure on TFI model, we follow the procedure of the CPMC algorithm on Hubbard model \cite{nguyen2014cpmc}. 
For the Hamiltonian \eqref{eq:TFI hamiltonian}, we use the Suzuki-Trotter approximation \cite{suzuki1976relationship,trotter1959product} to approximate $e^{-\Delta\tau H}$:
\begin{equation}\label{eq:trotter approximation}
    e^{-\Delta\tau H} = e^{-\Delta\tau H_1/2} e^{-\Delta\tau H_2} e^{-\Delta\tau H_1/2} + O(\Delta\tau^3 ):= B_{H_1/2}B_{H_2}B_{H_1/2}  + O(\Delta\tau^3 ), 
\end{equation}
where we refer to $B_{H_1/2}$ and $B_{H_2}$ as the one-body and two-body propagator, respectively. For the one-body propagator, it holds
\begin{equation*}
    B_{H_1/2} = e^{g\Delta\tau\sigma^x/2} \otimes e^{g\Delta\tau\sigma^x/2} \otimes \cdots \otimes e^{g\Delta\tau\sigma^x/2}. 
\end{equation*}
Now we need to decompose the two-body propagator $B_{H_2}$ in order to have the decomposition in \eqref{eq:operator decomposition}. The trick is to use a discrete Hubbard-Stratonovich transformation \cite{ulmke2000auxiliary}:
\begin{equation}\label{eq:HS transformation}
    e^{\Delta \tau \sigma_i^z\sigma_j^z} = e^{-\Delta\tau}\frac{1}{2}\sum_{x_{ij}=\pm1}  e^{x_{ij}\lambda_{ij}(\sigma_i^z+\sigma_j^z)}, 
\end{equation}
where the constants $\lambda_{ij}$ are given by $\cosh{2\lambda_{ij}} = e^{2\Delta t}$, and each bond $\langle i,j \rangle$ is corresponded with an auxiliary field $x_{ij}$. We rewrite \eqref{eq:HS transformation} as
\begin{equation*}
    e^{\Delta \tau \sigma_i^z\sigma_j^z} = \sum_{x_{ij}=\pm1} p(x_{ij})b(x_{ij}), 
\end{equation*}
where we treat $x_{ij}$ as a random variable and $p(x_{ij})=1/2$ for $x_{ij}=\pm1$ as its probability distribution function, and 
\begin{equation*}
    b(x_{ij}) = e^{-\Delta\tau} e^{x_{ij}\lambda_{ij}(\sigma_i^z+\sigma_j^z)} = e^{-\Delta\tau}e^{x_{ij}\lambda_{ij}\sigma_i^z}e^{x_{ij}\lambda_{ij}\sigma_j^z}
\end{equation*}
is a product of two separable operators, thus it returns another separable state when acting on one. The two-body propagator can then be rewritten as
\begin{equation}\label{eq:hs general}
    B_{H_2} = e^{\Delta\tau\sum_{\braket{i,j}}\sigma_i^z\sigma_j^z} = \prod_{\braket{i,j}}e^{\Delta t \sigma_i^z\sigma_j^z} = \prod_{\braket{i,j}}\sum_{x_{ij}=\pm1} p(x_{ij})b(x_{ij}),
\end{equation}
which yields the desired decomposition \eqref{eq:operator decomposition}. Here, the matrices $\sigma_i^z\sigma_j^z$ are diagonal, thus the equality is exact in \eqref{eq:hs general}, and no extra Trotter error will be introduced. Since the Hamiltonian contains two different terms $B_{H_2}$ and $B_{H_1/2}$, to make cp-AFQMC more efficient, we modify Algorithm~\ref{alg:cp-AFQMC} slightly as follows. 

\begin{algorithm}[!htbp]
\caption{cp-AFQMC on the TFI system}\label{alg:cp-AFQMC 2}
\begin{algorithmic}
\Require 
Hamiltonians $H_1$ and $H_2$; initial walkers $\widehat \Psi^{(0)} = \sum_k w_k^{(0)} \frac{\phi_k^{(0)}}{\braket{\Psi_\mathrm{tr}, \phi_k^{(0)}}}$; trial wavefunction $\Psi_\mathrm{tr}$, step size $\Delta\tau$, number of iteration steps $M$. 
\Ensure Walkers $\widehat \Psi^{(M)} = \sum_k w_k^{(M)} \frac{\phi_k^{(M)}}{\braket{\Psi_\mathrm{tr}, \phi_k^{(M)}}}$
\For{$n=0,1,\cdots,M-1$}
\begin{itemize}
\item 
For walkers with positive weights: 
\begin{enumerate}[label=\arabic*., ref=\arabic*, leftmargin=1cm]
\item Let $\phi^{(n)\prime}_k \leftarrow B_{H_1/2}\phi^{(n)}_k$, $w_k^{(n)}\leftarrow w_k^{(n)} \frac{\max\{\braket{\Psi_\mathrm{tr}, \phi^{(n)\prime}_k},0\}}{\braket{\Psi_\mathrm{tr}, \phi^{(n)}_k}}$.
\item Update $\phi_k^{(n)\prime}$ to $\phi_k^{(n)\prime\prime}$ by applying $B_{H_2}$ to $\phi_k^{(n)\prime}$ stochastically using the decomposition in  \eqref{eq:hs general}. Update weights $w_k^{(n)}$ to $w_k^{(n+1)}$ as in \eqref{eq:tfi weight update}. 
\item Let $\phi^{(n+1)}_{k}\leftarrow B_{H_1/2}\phi^{(n)\prime\prime}_k$, $w_k^{(n+1)}\leftarrow w_k^{(n+1)} \frac{\max\{\braket{\Psi_\mathrm{tr}, \phi^{(n+1)}_k},0\}}{\braket{\Psi_\mathrm{tr}, \phi^{(n)\prime\prime}_k}}$.
\end{enumerate}

\item 
Periodically estimate the ground-state energy as in \eqref{eq:mixed energy estimator}. 

\item 
Periodically apply population control as in Algorithm~\ref{alg:ppl control}. 

\end{itemize}

\EndFor

\end{algorithmic}
\end{algorithm}

We remark that all of the entries in both the matrices $B_{H_1/2}$ and $B_{H_2}$ are non-negative. Thus, if we apply $B_{H_1/2}B_{H_2}B_{H_1/2}$ to a wavefunction with all the entries being positive, it returns another all-positive-entry function. Therefore, by choosing a trial function with all positive entries, the walkers will remain such a form, thus it always holds $\braket{\Psi_\mathrm{tr},\phi_k^{(n)}}>0$. The trial wavefunction \eqref{eq:initial trial} satisfies the condition, thus there will be no sign problem if one runs cp-AFQMC with this fixed trial function.

We now detail how to apply $B_{H_2}$ to a walker $\phi_k^{(n)\prime}$ stochastically. We sample $x_{ij}$ for each $\langle i,j\rangle$ one at a time, and apply a single $b(x_{ij})$ to $\phi^{(n)\prime}_k$. For example, for a single pair $\langle i,j\rangle$, if at least one $x_{ij}$ satisfies $\langle \Psi_\mathrm{tr}, b(x_{ij})\phi_k^{(n)\prime} \rangle > 0$, then we define the probability distribution function
\begin{equation*}
 q_{ij,k}^{(n)}(x_{ij}) = \frac{1}{\mathcal{N}^{(n)}_{ij,k}}\frac{\max\{\langle \Psi_\mathrm{tr}, b(x_{ij})\phi_k^{(n)\prime} \rangle, 0\}}{\braket{\Psi_\mathrm{tr}, \phi_k^{(n)\prime}}}p(x_{ij}). 
\end{equation*}
The update rule for the weight and the walker becomes
\begin{equation}\label{eq:tfi weight update}
    w^{(n+1)}_{k} = w_k^{(n)} \mathcal{N}^{(n)}_{ij,k}. 
\end{equation}
and 
\begin{equation*}
\phi^{(n)\prime\prime}_{k}= b(x^{(n)}_{ij,k})\phi_k^{(n)\prime}, x_{ij,k}^{(n)} \sim q_{ij,k}^{(n)}. 
\end{equation*}
If $\langle \Psi_\mathrm{tr}, b(x_{ij})\phi_k^{(n)\prime} \rangle \le 0$ for both $x_{ij}=\pm1$, then we remove this walker by setting $w_k^{(n+1)}=0$. 
This procedure is then repeated for all pairs $\langle i, j\rangle$.

\section{Systematic Error of cp-AFQMC on TFI Systems}\label{sec:sanity check}

In this section, we study the origin of the bias in cp-AFQMC, using the TFI system with $16$ spins in 1D as an example. As mentioned in Section~\ref{sec:numerical results}, the TFI system is free of sign problem with a properly chosen trial wavefunction, but there is still a systematic error caused by the fixed imaginary-time step $\Delta\tau$ and the finite population size $N$. 

At each iteration step, a Trotter error $O(\Delta\tau^3)$ is introduced, as shown in \eqref{eq:trotter approximation}. For a fixed imaginary-time, one needs to run $O(1/\Delta\tau)$ steps of random walk. Thus, the total error in terms of the energy estimation accumulates to $O(\Delta\tau^2)$. 
Ideally, the Trotter error will vanish as the imaginary-time step $\Delta\tau$ goes to $0$. However, this is not always the case as shown in Figure~\ref{fig:trotter error}. With $2000$ walkers, the estimated energy indeed converges to the ground-state energy as $\Delta\tau$ goes to $0$. If we use $200$ walkers, however, there is still a quite obvious bias even when $\Delta\tau$ becomes $0$. Moreover, with fewer walkers, the bias is consistently larger with any imaginary-time step $\Delta\tau$. 

\begin{figure}[!htbp]
\centering
\subfigure[Trotter error]{
\begin{minipage}{0.49\textwidth}
    \centering
    \includegraphics[height=4cm, width=6cm]{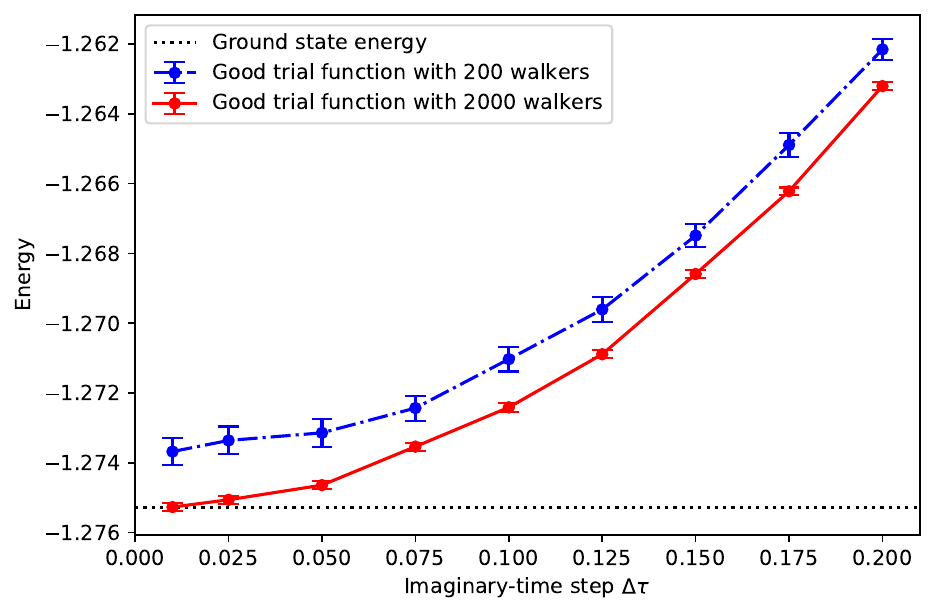}
    \label{fig:trotter error}
\end{minipage}
}
\hspace{-1.5cm}
\subfigure[Finite-population error]{
\begin{minipage}{0.49\textwidth}
    \centering
    \includegraphics[height=4cm, width=6cm]{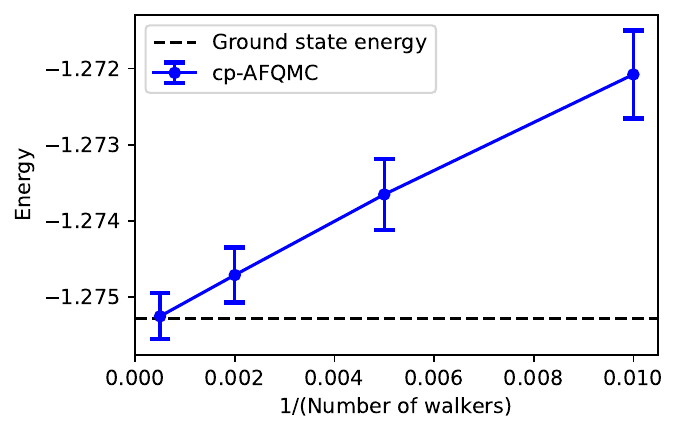}
    \label{fig:population size}
\end{minipage}
}
\vspace{-0.4cm}
\caption{Results of cp-AFQMC on $16$-spin system in 1D. (a): estimated energy with different imaginary-time step $\Delta\tau$; (b): estimated energy with different population size. }
\end{figure}

To further study how the error depends on the number of walkers, we fix the imaginary-time step to $\Delta\tau=0.01$, and run cp-AFQMC with different population size. As shown in Figure~\ref{fig:population size}, the error is roughly inversely proportional to the population size. The bias is introduced by the population control procedure as in Algorithm~\ref{alg:ppl control}. To be more specific, if we look at each single walker, it breaks into a random number of copies, with the average number of copies matching its weight. However, to maintain a constant population size, the number of copies must be adjusted based on the total, which introduces dependency among walkers. This dependency leads to a bias in the energy estimation.

\bibliographystyle{unsrt}  
\bibliography{library}

@article{ulmke2000auxiliary,
  title={Auxiliary-field Monte Carlo for quantum spin and boson systems},
  author={Ulmke, Martin and Scalettar, RT},
  journal={Physical Review B},
  volume={61},
  number={14},
  pages={9607},
  year={2000},
  publisher={APS}
}

@incollection{mps-1,
  title={Valence bond ground states in isotropic quantum antiferromagnets},
  author={Affleck, Ian and Kennedy, Tom and Lieb, Elliott H and Tasaki, Hal},
  booktitle={Condensed matter physics and exactly soluble models},
  pages={253--304},
  year={1988},
  publisher={Springer}
}

@article{tt-decomposition,
  title={Tensor-train decomposition},
  author={Oseledets, Ivan V},
  journal={SIAM Journal on Scientific Computing},
  volume={33},
  number={5},
  pages={2295--2317},
  year={2011},
  publisher={SIAM}
}

@article{mps-3,
  title={Density matrix formulation for quantum renormalization groups},
  author={White, Steven R},
  journal={Physical review letters},
  volume={69},
  number={19},
  pages={2863},
  year={1992},
  publisher={APS}
}

@article{nguyen2014cpmc,
  title={Cpmc-lab: A matlab package for constrained path monte carlo calculations},
  author={Nguyen, Huy and Shi, Hao and Xu, Jie and Zhang, Shiwei},
  journal={Computer Physics Communications},
  volume={185},
  number={12},
  pages={3344--3357},
  year={2014},
  publisher={Elsevier}
}

@article{chen2023combining,
  title={Combining particle and tensor-network methods for partial differential equations via sketching},
  author={Chen, Yian and Khoo, Yuehaw},
  journal={arXiv preprint arXiv:2305.17884},
  year={2023}
}

@article{fisher1998distributions,
  title={Distributions of gaps and end-to-end correlations in random transverse-field Ising spin chains},
  author={Fisher, Daniel S and Young, AP},
  journal={Physical Review B},
  volume={58},
  number={14},
  pages={9131},
  year={1998},
  publisher={APS}
}

@book{chakrabarti2008quantum,
  title={Quantum Ising phases and transitions in transverse Ising models},
  author={Chakrabarti, Bikas K and Dutta, Amit and Sen, Parongama},
  volume={41},
  year={2008},
  publisher={Springer Science \& Business Media}
}

@article{zhang1997constrained,
  title={Constrained path Monte Carlo method for fermion ground states},
  author={Zhang, Shiwei and Carlson, Joseph and Gubernatis, James E},
  journal={Physical Review B},
  volume={55},
  number={12},
  pages={7464},
  year={1997},
  publisher={APS}
}

@article{chang2008spatially,
  title={Spatially inhomogeneous phase in the two-dimensional repulsive Hubbard model},
  author={Chang, Chia-Chen and Zhang, Shiwei},
  journal={Physical Review B},
  volume={78},
  number={16},
  pages={165101},
  year={2008},
  publisher={APS}
}

@article{PhysRevD.24.2278,
  title = {Monte Carlo calculations of coupled boson-fermion systems. I},
  author = {Blankenbecler, R. and Scalapino, D. J. and Sugar, R. L.},
  journal = {Phys. Rev. D},
  volume = {24},
  issue = {8},
  pages = {2278--2286},
  numpages = {0},
  year = {1981},
  month = {Oct},
  publisher = {American Physical Society},
  doi = {10.1103/PhysRevD.24.2278},
  url = {https://link.aps.org/doi/10.1103/PhysRevD.24.2278}
}

@article{sugiyama1986auxiliary,
  title={Auxiliary field Monte-Carlo for quantum many-body ground states},
  author={Sugiyama, G and Koonin, SE},
  journal={Annals of Physics},
  volume={168},
  number={1},
  pages={1--26},
  year={1986},
  publisher={Elsevier}
}

@article{PhysRevB.40.506,
  title = {Numerical study of the two-dimensional Hubbard model},
  author = {White, S. R. and Scalapino, D. J. and Sugar, R. L. and Loh, E. Y. and Gubernatis, J. E. and Scalettar, R. T.},
  journal = {Phys. Rev. B},
  volume = {40},
  issue = {1},
  pages = {506--516},
  numpages = {0},
  year = {1989},
  month = {Jul},
  publisher = {American Physical Society},
  doi = {10.1103/PhysRevB.40.506},
  url = {https://link.aps.org/doi/10.1103/PhysRevB.40.506}
}

@article{RevModPhys.83.349,
  title = {Continuous-time Monte Carlo methods for quantum impurity models},
  author = {Gull, Emanuel and Millis, Andrew J. and Lichtenstein, Alexander I. and Rubtsov, Alexey N. and Troyer, Matthias and Werner, Philipp},
  journal = {Rev. Mod. Phys.},
  volume = {83},
  issue = {2},
  pages = {349--404},
  numpages = {0},
  year = {2011},
  month = {May},
  publisher = {American Physical Society},
  doi = {10.1103/RevModPhys.83.349},
  url = {https://link.aps.org/doi/10.1103/RevModPhys.83.349}
}

@article{PhysRevB.41.9301,
  title = {Sign problem in the numerical simulation of many-electron systems},
  author = {Loh, E. Y. and Gubernatis, J. E. and Scalettar, R. T. and White, S. R. and Scalapino, D. J. and Sugar, R. L.},
  journal = {Phys. Rev. B},
  volume = {41},
  issue = {13},
  pages = {9301--9307},
  numpages = {0},
  year = {1990},
  month = {May},
  publisher = {American Physical Society},
  doi = {10.1103/PhysRevB.41.9301},
  url = {https://link.aps.org/doi/10.1103/PhysRevB.41.9301}
}

@article{PhysRevLett.94.170201,
  title = {Computational Complexity and Fundamental Limitations to Fermionic Quantum Monte Carlo Simulations},
  author = {Troyer, Matthias and Wiese, Uwe-Jens},
  journal = {Phys. Rev. Lett.},
  volume = {94},
  issue = {17},
  pages = {170201},
  numpages = {4},
  year = {2005},
  month = {May},
  publisher = {American Physical Society},
  doi = {10.1103/PhysRevLett.94.170201},
  url = {https://link.aps.org/doi/10.1103/PhysRevLett.94.170201}
}

@article{PhysRevB.94.235119,
  title = {Coupling quantum Monte Carlo and independent-particle calculations: Self-consistent constraint for the sign problem based on the density or the density matrix},
  author = {Qin, Mingpu and Shi, Hao and Zhang, Shiwei},
  journal = {Phys. Rev. B},
  volume = {94},
  issue = {23},
  pages = {235119},
  numpages = {5},
  year = {2016},
  month = {Dec},
  publisher = {American Physical Society},
  doi = {10.1103/PhysRevB.94.235119},
  url = {https://link.aps.org/doi/10.1103/PhysRevB.94.235119}
}

@article{zheng2017stripe,
  title={Stripe order in the underdoped region of the two-dimensional Hubbard model},
  author={Zheng, Bo-Xiao and Chung, Chia-Min and Corboz, Philippe and Ehlers, Georg and Qin, Ming-Pu and Noack, Reinhard M and Shi, Hao and White, Steven R and Zhang, Shiwei and Chan, Garnet Kin-Lic},
  journal={Science},
  volume={358},
  number={6367},
  pages={1155--1160},
  year={2017},
  publisher={American Association for the Advancement of Science}
}

@article{PhysRevB.99.045108,
  title = {Finite-temperature auxiliary-field quantum Monte Carlo: Self-consistent constraint and systematic approach to low temperatures},
  author = {He, Yuan-Yao and Qin, Mingpu and Shi, Hao and Lu, Zhong-Yi and Zhang, Shiwei},
  journal = {Phys. Rev. B},
  volume = {99},
  issue = {4},
  pages = {045108},
  numpages = {15},
  year = {2019},
  month = {Jan},
  publisher = {American Physical Society},
  doi = {10.1103/PhysRevB.99.045108},
  url = {https://link.aps.org/doi/10.1103/PhysRevB.99.045108}
}

@article{PhysRevResearch.3.013065,
  title = {Pseudo-BCS wave function from density matrix decomposition: Application in auxiliary-field quantum Monte Carlo},
  author = {Xiao, Zhi-Yu and Shi, Hao and Zhang, Shiwei},
  journal = {Phys. Rev. Res.},
  volume = {3},
  issue = {1},
  pages = {013065},
  numpages = {10},
  year = {2021},
  month = {Jan},
  publisher = {American Physical Society},
  doi = {10.1103/PhysRevResearch.3.013065},
  url = {https://link.aps.org/doi/10.1103/PhysRevResearch.3.013065}
}

@article{shi2021some,
  title={Some recent developments in auxiliary-field quantum Monte Carlo for real materials},
  author={Shi, Hao and Zhang, Shiwei},
  journal={The Journal of chemical physics},
  volume={154},
  number={2},
  year={2021},
  publisher={AIP Publishing}
}

@article{PhysRevLett.69.2863,
  title = {Density matrix formulation for quantum renormalization groups},
  author = {White, Steven R.},
  journal = {Phys. Rev. Lett.},
  volume = {69},
  issue = {19},
  pages = {2863--2866},
  numpages = {0},
  year = {1992},
  month = {Nov},
  publisher = {American Physical Society},
  doi = {10.1103/PhysRevLett.69.2863},
  url = {https://link.aps.org/doi/10.1103/PhysRevLett.69.2863}
}

@article{hur2023generative,
  title={Generative modeling via tensor train sketching},
  author={Hur, Yoonhaeng and Hoskins, Jeremy G and Lindsey, Michael and Stoudenmire, E Miles and Khoo, Yuehaw},
  journal={Applied and Computational Harmonic Analysis},
  volume={67},
  pages={101575},
  year={2023},
  publisher={Elsevier}
}

@article{suzuki1976relationship,
  title={Relationship between d-dimensional quantal spin systems and (d+ 1)-dimensional ising systems: Equivalence, critical exponents and systematic approximants of the partition function and spin correlations},
  author={Suzuki, Masuo},
  journal={Progress of theoretical physics},
  volume={56},
  number={5},
  pages={1454--1469},
  year={1976},
  publisher={Oxford University Press}
}

@article{trotter1959product,
  title={On the product of semi-groups of operators},
  author={Trotter, Hale F},
  journal={Proceedings of the American Mathematical Society},
  volume={10},
  number={4},
  pages={545--551},
  year={1959}
}

@article{al2007study,
  title={A study of H+ H2 and several H-bonded molecules by phaseless auxiliary-field quantum Monte Carlo with plane wave and Gaussian basis sets},
  author={Al-Saidi, WA and Krakauer, Henry and Zhang, Shiwei},
  journal={The Journal of chemical physics},
  volume={126},
  number={19},
  year={2007},
  publisher={AIP Publishing}
}

@article{purwanto2009excited,
  title={Excited state calculations using phaseless auxiliary-field quantum Monte Carlo: Potential energy curves of low-lying C2 singlet states},
  author={Purwanto, Wirawan and Zhang, Shiwei and Krakauer, Henry},
  journal={The Journal of chemical physics},
  volume={130},
  number={9},
  year={2009},
  publisher={AIP Publishing}
}

@article{purwanto2008eliminating,
  title={Eliminating spin contamination in auxiliary-field quantum Monte Carlo: Realistic potential energy curve of F2},
  author={Purwanto, Wirawan and Al-Saidi, WA and Krakauer, Henry and Zhang, Shiwei},
  journal={The Journal of chemical physics},
  volume={128},
  number={11},
  year={2008},
  publisher={AIP Publishing}
}

@article{PhysRevLett.90.136401,
  title = {Quantum Monte Carlo Method using Phase-Free Random Walks with Slater Determinants},
  author = {Zhang, Shiwei and Krakauer, Henry},
  journal = {Phys. Rev. Lett.},
  volume = {90},
  issue = {13},
  pages = {136401},
  numpages = {4},
  year = {2003},
  month = {Apr},
  publisher = {American Physical Society},
  doi = {10.1103/PhysRevLett.90.136401},
  url = {https://link.aps.org/doi/10.1103/PhysRevLett.90.136401}
}

@article{PhysRevB.97.235127,
  title = {Magnetic orders in the hole-doped three-band Hubbard model: Spin spirals, nematicity, and ferromagnetic domain walls},
  author = {Chiciak, Adam and Vitali, Ettore and Shi, Hao and Zhang, Shiwei},
  journal = {Phys. Rev. B},
  volume = {97},
  issue = {23},
  pages = {235127},
  numpages = {11},
  year = {2018},
  month = {Jun},
  publisher = {American Physical Society},
  doi = {10.1103/PhysRevB.97.235127},
  url = {https://link.aps.org/doi/10.1103/PhysRevB.97.235127}
}

@article{PhysRevA.86.053606,
  title = {Finite-temperature auxiliary-field quantum Monte Carlo technique for Bose-Fermi mixtures},
  author = {Rubenstein, Brenda M. and Zhang, Shiwei and Reichman, David R.},
  journal = {Phys. Rev. A},
  volume = {86},
  issue = {5},
  pages = {053606},
  numpages = {14},
  year = {2012},
  month = {Nov},
  publisher = {American Physical Society},
  doi = {10.1103/PhysRevA.86.053606},
  url = {https://link.aps.org/doi/10.1103/PhysRevA.86.053606}
}

@article{PhysRevLett.119.265301,
  title = {Ultracold Atoms in a Square Lattice with Spin-Orbit Coupling: Charge Order, Superfluidity, and Topological Signatures},
  author = {Rosenberg, Peter and Shi, Hao and Zhang, Shiwei},
  journal = {Phys. Rev. Lett.},
  volume = {119},
  issue = {26},
  pages = {265301},
  numpages = {5},
  year = {2017},
  month = {Dec},
  publisher = {American Physical Society},
  doi = {10.1103/PhysRevLett.119.265301},
  url = {https://link.aps.org/doi/10.1103/PhysRevLett.119.265301}
}

@article{PhysRevB.99.165116,
  title = {Metal-insulator transition in the ground state of the three-band Hubbard model at half filling},
  author = {Vitali, Ettore and Shi, Hao and Chiciak, Adam and Zhang, Shiwei},
  journal = {Phys. Rev. B},
  volume = {99},
  issue = {16},
  pages = {165116},
  numpages = {5},
  year = {2019},
  month = {Apr},
  publisher = {American Physical Society},
  doi = {10.1103/PhysRevB.99.165116},
  url = {https://link.aps.org/doi/10.1103/PhysRevB.99.165116}
}

@article{PhysRevB.102.041106,
  title = {Plaquette versus ordinary $d$-wave pairing in the ${t}^{\ensuremath{'}}$-Hubbard model on a width-4 cylinder},
  author = {Chung, Chia-Min and Qin, Mingpu and Zhang, Shiwei and Schollw\"ock, Ulrich and White, Steven R.},
  collaboration = {The Simons Collaboration on the Many-Electron Problem},
  journal = {Phys. Rev. B},
  volume = {102},
  issue = {4},
  pages = {041106},
  numpages = {6},
  year = {2020},
  month = {Jul},
  publisher = {American Physical Society},
  doi = {10.1103/PhysRevB.102.041106},
  url = {https://link.aps.org/doi/10.1103/PhysRevB.102.041106}
}

@article{PhysRevB.102.214512,
  title = {Magnetic and charge orders in the ground state of the Emery model: Accurate numerical results},
  author = {Chiciak, Adam and Vitali, Ettore and Zhang, Shiwei},
  journal = {Phys. Rev. B},
  volume = {102},
  issue = {21},
  pages = {214512},
  numpages = {15},
  year = {2020},
  month = {Dec},
  publisher = {American Physical Society},
  doi = {10.1103/PhysRevB.102.214512},
  url = {https://link.aps.org/doi/10.1103/PhysRevB.102.214512}
}

@article{PhysRevB.103.115123,
  title = {Constrained-path auxiliary-field quantum Monte Carlo for coupled electrons and phonons},
  author = {Lee, Joonho and Zhang, Shiwei and Reichman, David R.},
  journal = {Phys. Rev. B},
  volume = {103},
  issue = {11},
  pages = {115123},
  numpages = {19},
  year = {2021},
  month = {Mar},
  publisher = {American Physical Society},
  doi = {10.1103/PhysRevB.103.115123},
  url = {https://link.aps.org/doi/10.1103/PhysRevB.103.115123}
}

@article{PhysRevResearch.4.013239,
  title = {Stripes and spin-density waves in the doped two-dimensional Hubbard model: Ground state phase diagram},
  author = {Xu, Hao and Shi, Hao and Vitali, Ettore and Qin, Mingpu and Zhang, Shiwei},
  journal = {Phys. Rev. Res.},
  volume = {4},
  issue = {1},
  pages = {013239},
  numpages = {10},
  year = {2022},
  month = {Mar},
  publisher = {American Physical Society},
  doi = {10.1103/PhysRevResearch.4.013239},
  url = {https://link.aps.org/doi/10.1103/PhysRevResearch.4.013239}
}

@article{PhysRevLett.128.203201,
  title = {Exotic Superfluid Phases in Spin-Polarized Fermi Gases in Optical Lattices},
  author = {Vitali, Ettore and Rosenberg, Peter and Zhang, Shiwei},
  journal = {Phys. Rev. Lett.},
  volume = {128},
  issue = {20},
  pages = {203201},
  numpages = {5},
  year = {2022},
  month = {May},
  publisher = {American Physical Society},
  doi = {10.1103/PhysRevLett.128.203201},
  url = {https://link.aps.org/doi/10.1103/PhysRevLett.128.203201}
}

@article{al2006auxiliary,
  title={Auxiliary-field quantum Monte Carlo calculations of molecular systems with a Gaussian basis},
  author={Al-Saidi, WA and Zhang, Shiwei and Krakauer, Henry},
  journal={The Journal of chemical physics},
  volume={124},
  number={22},
  year={2006},
  publisher={AIP Publishing}
}

@article{PhysRevLett.114.226401,
  title = {Quantum Monte Carlo Calculations in Solids with Downfolded Hamiltonians},
  author = {Ma, Fengjie and Purwanto, Wirawan and Zhang, Shiwei and Krakauer, Henry},
  journal = {Phys. Rev. Lett.},
  volume = {114},
  issue = {22},
  pages = {226401},
  numpages = {5},
  year = {2015},
  month = {Jun},
  publisher = {American Physical Society},
  doi = {10.1103/PhysRevLett.114.226401},
  url = {https://link.aps.org/doi/10.1103/PhysRevLett.114.226401}
}

@article{shee2017chemical,
  title={Chemical transformations approaching chemical accuracy via correlated sampling in auxiliary-field quantum Monte Carlo},
  author={Shee, James and Zhang, Shiwei and Reichman, David R and Friesner, Richard A},
  journal={Journal of chemical theory and computation},
  volume={13},
  number={6},
  pages={2667--2680},
  year={2017},
  publisher={ACS Publications}
}

@article{motta2018communication,
  title={Communication: Calculation of interatomic forces and optimization of molecular geometry with auxiliary-field quantum Monte Carlo},
  author={Motta, Mario and Zhang, Shiwei},
  journal={The Journal of Chemical Physics},
  volume={148},
  number={18},
  year={2018},
  publisher={AIP Publishing}
}

@article{shee2018phaseless,
  title={Phaseless auxiliary-field quantum Monte Carlo on graphical processing units},
  author={Shee, James and Arthur, Evan J and Zhang, Shiwei and Reichman, David R and Friesner, Richard A},
  journal={Journal of chemical theory and computation},
  volume={14},
  number={8},
  pages={4109--4121},
  year={2018},
  publisher={ACS Publications}
}

@article{shee2019singlet,
  title={Singlet--triplet energy gaps of organic biradicals and polyacenes with auxiliary-field quantum Monte Carlo},
  author={Shee, James and Arthur, Evan J and Zhang, Shiwei and Reichman, David R and Friesner, Richard A},
  journal={Journal of chemical theory and computation},
  volume={15},
  number={9},
  pages={4924--4932},
  year={2019},
  publisher={ACS Publications}
}

@article{lee2020performance,
  title={The performance of phaseless auxiliary-field quantum Monte Carlo on the ground state electronic energy of benzene},
  author={Lee, Joonho and Malone, Fionn D and Reichman, David R},
  journal={The Journal of Chemical Physics},
  volume={153},
  number={12},
  year={2020},
  publisher={AIP Publishing}
}

@article{sachdev1999quantum,
  title={Quantum phase transitions},
  author={Sachdev, Subir},
  journal={Physics world},
  volume={12},
  number={4},
  pages={33},
  year={1999},
  publisher={IOP Publishing}
}

@article{pfeuty1970one,
  title={The one-dimensional Ising model with a transverse field},
  author={Pfeuty, Pierre},
  journal={ANNALS of Physics},
  volume={57},
  number={1},
  pages={79--90},
  year={1970},
  publisher={Elsevier}
}

@article{wang2015fast,
  title={Fast and guaranteed tensor decomposition via sketching},
  author={Wang, Yining and Tung, Hsiao-Yu and Smola, Alexander J and Anandkumar, Anima},
  journal={Advances in neural information processing systems},
  volume={28},
  year={2015}
}

@article{ahle2019almost,
  title={Almost optimal tensor sketch},
  author={Ahle, Thomas D and Knudsen, Jakob BT},
  journal={arXiv preprint arXiv:1909.01821},
  year={2019}
}

@article{peng2023generative,
  title={Generative modeling via hierarchical tensor sketching},
  author={Peng, Yifan and Chen, Yian and Stoudenmire, E Miles and Khoo, Yuehaw},
  journal={arXiv preprint arXiv:2304.05305},
  year={2023}
}

@article{fishman2022itensor,
  title={The ITensor software library for tensor network calculations},
  author={Fishman, Matthew and White, Steven and Stoudenmire, Edwin},
  journal={SciPost Physics Codebases},
  pages={004},
  year={2022}
}

@article{blote2002cluster,
  title={Cluster Monte Carlo simulation of the transverse Ising model},
  author={Bl{\"o}te, Henk WJ and Deng, Youjin},
  journal={Physical Review E},
  volume={66},
  number={6},
  pages={066110},
  year={2002},
  publisher={APS}
}

@article{gull1993imaginary,
  title={Imaginary numbers are not real—the geometric algebra of spacetime},
  author={Gull, Stephen and Lasenby, Anthony and Doran, Chris},
  journal={Foundations of Physics},
  volume={23},
  number={9},
  pages={1175--1201},
  year={1993},
  publisher={Springer}
}

@article{lubasch2014algorithms,
  title={Algorithms for finite projected entangled pair states},
  author={Lubasch, Michael and Cirac, J Ignacio and Banuls, Mari-Carmen},
  journal={Physical Review B},
  volume={90},
  number={6},
  pages={064425},
  year={2014},
  publisher={APS}
}

@article{wouters2014projector,
  title={Projector quantum Monte Carlo with matrix product states},
  author={Wouters, Sebastian and Verstichel, Brecht and Van Neck, Dimitri and Chan, Garnet Kin-Lic},
  journal={Physical Review B},
  volume={90},
  number={4},
  pages={045104},
  year={2014},
  publisher={APS}
}

@article{whitlock1979properties,
  title={Properties of liquid and solid He 4},
  author={Whitlock, P\_A and Ceperley, DM and Chester, GV and Kalos, MH},
  journal={Physical Review B},
  volume={19},
  number={11},
  pages={5598},
  year={1979},
  publisher={APS}
}

@article{politzer2018hellmann,
  title={The Hellmann-Feynman theorem: a perspective},
  author={Politzer, Peter and Murray, Jane S},
  journal={Journal of molecular modeling},
  volume={24},
  pages={1--7},
  year={2018},
  publisher={Springer}
}

@article{sowell1989taxicab,
  title={Taxicab geometry—a new slant},
  author={Sowell, Katye O},
  journal={Mathematics Magazine},
  volume={62},
  number={4},
  pages={238--248},
  year={1989},
  publisher={Taylor \& Francis}
}

@article{purwanto2004-brute-force-estimator,
  title = {Quantum Monte Carlo method for the ground state of many-boson systems},
  author = {Purwanto, Wirawan and Zhang, Shiwei},
  journal = {Phys. Rev. E},
  volume = {70},
  issue = {5},
  pages = {056702},
  numpages = {18},
  year = {2004},
  month = {Nov},
  publisher = {American Physical Society},
  doi = {10.1103/PhysRevE.70.056702},
  url = {https://link.aps.org/doi/10.1103/PhysRevE.70.056702}
}

@article{drineas2006fast,
  title={Fast Monte Carlo algorithms for matrices II: Computing a low-rank approximation to a matrix},
  author={Drineas, Petros and Kannan, Ravi and Mahoney, Michael W},
  journal={SIAM Journal on computing},
  volume={36},
  number={1},
  pages={158--183},
  year={2006},
  publisher={SIAM}
}

\end{document}